\documentclass[]{article}

\usepackage{graphicx}
\usepackage{amsmath}
\usepackage{amssymb}
\usepackage{amsthm} 
\usepackage{bm}
\usepackage{enumerate}
\usepackage{color}
\usepackage{mathdots}
\usepackage{sectsty}
\usepackage{hyperref}
\hypersetup{hidelinks}
\usepackage{tikz}
\usepackage{caption}
\usepackage{adjustbox}
\usepackage{yhmath}
\usepackage{tikz-cd}
\usepackage{mathrsfs}
\usepackage{quiver}

\sectionfont{\scshape\centering\fontsize{12}{14}\selectfont}
\subsectionfont{\scshape\fontsize{12}{14}\selectfont}
\usepackage{fancyhdr}

\newcommand\shorttitle{Quasicoherent sheaves for dagger analytic geometry}
\newcommand\authors{Arun Soor}

\newtheorem{thm}{Theorem}[section]
\newtheorem{cor}[thm]{Corollary}
\newtheorem{lem}[thm]{Lemma}
\newtheorem{defn}[thm]{Definition}
\newtheorem{rmk}[thm]{Remark}
\newtheorem{example}[thm]{Example}
\newtheorem{prop}[thm]{Proposition}

\newtheorem*{claim*}{Claim}

\newcommand*{\sheafhom}{\mathcal{H}\kern -.5pt om}

\newcommand{\Addresses}{{
  \bigskip
  \footnotesize

  Arun Soor, \textsc{Mathematical Institute, University of Oxford,
    Andrew Wiles Building, Radcliffe Observatory Quarter, Woodstock Road, Oxford, United Kingdom OX2 6GG}\par\nopagebreak
  \textit{E-mail address}, Arun Soor, \texttt{\href{mailto:arun.soor@maths.ox.ac.uk}{arun.soor@maths.ox.ac.uk} }}}

\title{\large \bf Quasicoherent sheaves for dagger analytic geometry}
\author{Arun Soor}
\date{}
\begin{document}
\maketitle
\abstract{We develop a theory of quasicoherent sheaves on dagger analytic varieties based on Ind-Banach spaces. We show that they satisfy descent in the analytic topology. We define compactly supported pushforwards and produce an adjunction $f_! \dashv f^!$, and produce an excision sequence for certain inclusions of admissible open subsets. Finally, we prove a Grothendieck duality theorem.}
\tableofcontents
\section{Introduction}
This work is borne from the author's desire for a good theory of quasicoherent sheaves in $p$-adic analytic geometry, whilst taking a somewhat elementary and classical approach. The groundbreaking works of \cite{andreychev_pseudocoherent_2021,CondensedMathematics,CondensedComplexGeometry,mann_p-adic_2022,Camargo_deRham}, represent several highly successful approaches in various contexts already. In contrast and complement to theirs, we try to present what can be done without the use of condensed mathematics, derived geometry, animation, categorified locales, or compactifications. Since we have decided not to use derived geometry, we cannot obtain a six-functor formalism in the sense of \cite{mann_p-adic_2022}, as we will not be able to obtain arbitrary base change for lower shriek functors. 

Let us first try to explain the obstructions to such a theory. It is well known due to the example of Gabber \cite[Example 2.1.6]{ConradRelativeAmpleness} that the na\"ive definition of a quasicoherent sheaf is ill-behaved in analytic geometry. The inherent pathology is the failure of flatness of localizations, with respect to the completed tensor product. To fix this, for our theory of quasicoherent sheaves we will need to glue categories of complexes of modules along derived localizations. However, mapping cones in triangulated categories fail to be functorial in general and so one cannot, for instance, define a ``glued sheaf" along two open subsets to be the ``fiber of the two restriction morphisms". This necessitates the introduction of higher algebra; instead of triangulated categories, we use stable $\infty$-categories; this is also a feature of \cite{andreychev_pseudocoherent_2021,mann_p-adic_2022,CondensedComplexGeometry}. We assume familiarity with the basics, and direct the reader to \cite{HigherAlgebra,HigherToposTheory,kerodon} for details.

Moreover, our categories of modules should take into account the analytic structure of the rings we work with, and so we should work with modules in, for instance, Banach, ind-Banach, or bornological spaces. These categories fail to be abelian. The observation of Schneiders and Prosmans \cite{schneiders_quasi-abelian_1999,ProsmansQuasiAbelienne,ProsmansHomological} was that, more often than not, they are still a very special kind of exact category known as a quasi-abelian category, amenable to homological algebra. In particular every quasi-abelian category admits a fully faithful embedding into an abelian category known as its left heart, such that the embedding induces an equivalence on derived categories. In \cite{kelly_homotopy_2021}, Kelly produced a projective model structure on unbounded chain complexes valued in certain exact categories. In \S\ref{sec:derivedquasiabelian}, we rely crucially on his work, and describe a mostly formal procedure to produce a stack of quasicoherent sheaves on a site whose underlying category is the opposite category of a certain category of monoids in a nice enough quasi-abelian category. Many of these general constructions have already appeared in \cite{toen_homotopical_2008,CondensedMathematics} and the recent work \cite{DAnG}.

Let us now explain our results. As our preferred theory of $p$-adic analytic geometry, we use Grosse-Kl\"onne's dagger analytic spaces \cite{GrosseKlonneRigid}, (see \S\ref{sec:quasicoherent} for a brief introduction). These are constructed in an entirely analogous way to Tate's rigid analytic spaces, but are locally modelled on algebras of overconvergent power series. Let $K/\mathbb{Q}_p$ be any complete field extension and let $A$ be a $K$-dagger algebra. We view $A$ as a monoid in $\mathsf{IndBan}_K$, the $\operatorname{Ind}$-category of $K$-Banach spaces. By the formalism of \S \ref{sec:derivedquasiabelian}, we obtain, by localization of model categories, a $\mathsf{CAlg}(\mathsf{Pr}^L_{\mathsf{st}})$-valued\footnote{Here $\mathsf{CAlg}(\mathsf{Pr}^L_{\mathsf{st}})$ is the category of presentably symmetric monoidal stable $\infty$-categories with colimit preserving symmetric monoidal functors between them.} prestack $\operatorname{QCoh}$ whose value on objects is given as\footnote{This coincides with the definition of quasicoherent sheaves as given in the overconvergent context of \cite{DAnG}.} 
\begin{equation}
    \operatorname{QCoh}(\operatorname{Sp}(A)) := \mathsf{Mod}_A (N(\operatorname{Ch}(\mathsf{IndBan}_K))[W^{-1}])
\end{equation}
and which sends $f: \operatorname{Sp}(A) \to \operatorname{Sp}(B)$ to the (derived) pullback $f^*$. In \S\ref{sec:coherentsheaves}, we define a sub-prestack $\operatorname{Coh}^- \subset \operatorname{QCoh}$, such that $\operatorname{Coh}^-(\operatorname{Sp}(A))$ is the full subcategory of $\operatorname{QCoh}(\operatorname{Sp}(A))$ on cohomologically bounded above complexes whose cohomology modules are finitely generated $A$-modules.  Let $\mathsf{Afnd}^\dagger_w$ be the big site of all dagger affinoid varieties with the weak G-topology and let $\mathsf{Rig}_s^\dagger$ be the big site of all dagger analytic varieties with the strong G-topology.
\begin{thm}\label{thm:intro1}(c.f. \S\ref{sec:quasicoherent}, \S\ref{sec:coherentsheaves}).
\begin{enumerate}[(i)]
    \item $\operatorname{QCoh}$ is a $\mathsf{CAlg}(\mathsf{Pr}^L_{\mathsf{st}})$-valued sheaf on $\mathsf{Afnd}^\dagger_w$ and $\operatorname{Coh}^-$ is a sheaf of stably symmetric monoidal $\infty$-categories on $\mathsf{Afnd}^\dagger_w$. Right Kan extension along $\mathsf{Afnd}^\dagger \to \mathsf{Rig}_s^\dagger$ makes $\operatorname{QCoh}$ into a $\mathsf{CAlg}(\mathsf{Pr}^L_{\mathsf{st}})$-valued sheaf on $\mathsf{Rig}_s^\dagger$ and $\operatorname{Coh}^-$ into a sheaf of stable symmetric monoidal $\infty$-categories on $\mathsf{Rig}_s^\dagger$.
    \item For a morphism $f: X \to Y$ of dagger analytic varieties, the induced pullback functor $f^*: \operatorname{QCoh}(Y) \to \operatorname{QCoh}(X)$ admits a right adjoint $f_*$. 
    \item If $f$ is qcqs then $f_*$ commutes with restrictions, commutes with all colimits and satisfies\footnote{We deliberately suppressed the fact that various functors are really derived.} the projection formula $f_* \widehat{\otimes}_Y \operatorname{id} \xrightarrow[]{\sim} f_*(\operatorname{id}\widehat{\otimes}_X f^*)$. 
    \item If $t: Y^\prime \to Y$ is a \emph{transversal morphism} (Definition \ref{defn:transversalglobal}) then the canonical morphism\footnote{Here $f^\prime$ (resp. $t^\prime$), is the base-change of $f$ (resp $t$), along $t$, (resp. $f$).} $t^*f_* \to f^\prime_*t^{\prime,*}$ is an equivalence. Every smooth morphism of dagger-analytic varieties is transversal. 
\end{enumerate}
\end{thm}
In \S\ref{sec:compactsupports}, for any morphism $f:X \to Y$ of dagger analytic varieties we define a certain poset $c(X/Y)$ of admissible open\footnote{``Admissible open" will always mean an admissible open subset in the strong G-topology.} subsets of $X$, and we say that $M \in \operatorname{QCoh}(X)$ is compactly supported (relative to $Y$) if there exists $Z \in c(X/Y)$ such that $i^*M \simeq 0$, where $i: X \setminus Z \to X$ is the inclusion. The poset $c(X/Y)$ should intuitively be thought of as a collection of ``wide open" subsets each of which can be covered by an admissible open subset which is quasi-compact over $Y$. We define $f_!$ as the left Kan extension of $f_*$ from the full subcategory of compactly supported (relative to $Y$) objects. If $f$ is quasi-compact then $f_! = f_*$. Then the following Theorem says that $f_!$ bootstraps the good properties that pushforwards along a quasi-compact morphism has, under the hypothesis that $f$ is \emph{exhaustible} - this simply means that $\{Z: Z \in c(X/Y)\}$ is an admissible cover of $X$. 
\begin{thm}\label{thm:intro2}(c.f. \S\ref{sec:compactsupports}).
Let $f: X \to Y$ be any morphism of dagger analytic varieties. 
\begin{enumerate}[(i)]
    \item The functor $f_!$ admits a right adjoint $f^!$. 
    \item Assume that $f$ is exhaustible. Then functor $f_!$ commutes with restrictions, satisfies transversal base-change, and satisfies the projection formula $f_! \widehat{\otimes}_Y \operatorname{id} \xrightarrow[]{\sim} f_!(\operatorname{id} \widehat{\otimes}_X f^*)$.
    \item Let $g: Y \to Z$ be a further morphism of dagger analytic varieties and assume that $gf$ is exhaustible. Then $(gf)_! \simeq g_! f_!$ if either $g$ is quasi-compact, or $g$ is partially proper and $Z$ is quasi-paracompact and separated.
\end{enumerate}
\end{thm}
Though we have stated Theorems \ref{thm:intro1} and \ref{thm:intro2} for dagger analytic varieties, they hold just as well for rigid analytic varieties. The following Theorem, whose proof is the reason we need to use dagger geometry over the usual rigid geometry, shows that there is an intricate connection between the topological properties of admissible open subsets of a dagger analytic variety $X$, and the categorical properties of the pushforward functors along those inclusions of admissible open subsets.
\begin{thm}(c.f. \S\ref{subsubsec:domainembeddings}). 
Let $X$ be a dagger analytic variety, and let $k:V \to X$ be the inclusion of an admissible open subset such that $k$ is a quasi-compact morphism. Assume that $U := X \setminus V$ is also an admissible open subset such that $h:U \to X$ is exhaustible. Then for any $M \in \operatorname{QCoh}(U), N \in \operatorname{QCoh}(X)$, there is a canonical fiber-cofiber sequence
\begin{equation}
\begin{aligned}
 h_! h^* N \to N \to k_*k^*N,
\end{aligned}
\end{equation}
which we call \emph{excision}. Moreover, there is an adjunction $h_! \dashv h^*$, so that $h^! \simeq h^*$ in this case. 
\end{thm}
In \S\ref{sec:grothdual} we attempt to understand the functor $f^!$ in certain situations, in other words, we investigate Grothendieck duality. We recall the definition of a \emph{partially-proper} morphism of dagger-analytic varieties from \cite[Definition 0.4.2]{HuberEtale}.
\begin{thm}\label{thm:introgrothdual}(c.f. \S\ref{sec:grothdual}).
\begin{enumerate}[(i)]
    \item Let $Y$ be a dagger analytic variety and let $f: \mathring{\mathbb{D}}^n_Y = \mathring{\mathbb{D}}^n_K \times Y \to Y$ be the projection from the open polydisk. Then $f^!$ commutes with restrictions and there is an equivalence $f^! \simeq f^*[n]$.
    \item Let $f:X \to Y$ be an exhaustible and partially-proper \'etale morphism of dagger-analytic varieties. Then $f^!$ commutes with restrictions and there is a canonical equivalence $f^! \xrightarrow[]{\sim} f^*$.
    \item Let $f: X \to Y$ be an exhaustible, smooth and partially proper morphism of dagger-analytic varieties. Then $f^!$ commutes with restrictions and there is a canonical equivalence $f^! \xrightarrow[]{\sim} f^!1_Y \widehat{\otimes}_X f^*$. 
\end{enumerate}
\end{thm}
We remark (c.f. Remark \ref{rmk:final}) that part (iii) of this Theorem includes the case of all smooth proper morphisms, and when $Y$ is affinoid includes the case of all smooth relative Stein spaces of \cite{vanderPutSerre}. 

It is not clear presently how Theorem \ref{thm:introgrothdual} is precisely related to classical Serre duality results for rigid analytic varieties, see for example \cite{ChiarellottoDuality,BeyerSerre,BeyerThesis,vanderPutSerre}. This would require an investigation of the ``Verdier isomorphism" $f^!1_Y \simeq \omega_{X/Y}[n]$, where $\omega_{X/Y} := \bigwedge^n\Omega^1_{X/Y}$ is the relative canonical bundle, and also a version of Kiehl's theorem \cite{kiehl_endlichkeitssatz_1967} on the coherence of proper direct images. We postpone this to future work.

\paragraph{Acknowledgements.} The author would like to thank Federico Bambozzi, Andreas Bode, Jack Kelly, Ken Lee, Kobi Kremnizer, and Finn Wiersig for their interest in this work, and for answering many questions. We would like to thank the authors of \cite{DAnG} for sharing their manuscript. We also thank Werner L\"utkebohmert for sharing the thesis of Peter Beyer. The author would like to thank his supervisor Konstantin Ardakov for his continuous support and interest in this project. The author would like to thank the anonymous referee for numerous detailed suggestions which greatly improved this manuscript. The author is currently a DPhil student in Oxford supported by an EPSRC scholarship. 
\section{The derived $\infty$-category of a quasi-abelian category}\label{sec:derivedquasiabelian}

We make extensive use of the theory of homotopical algebra in quasi-abelian categories as developed in \cite{kelly_homotopy_2021} and \cite{schneiders_quasi-abelian_1999}. We assume familiarity with the basics of higher algebra and model categories, as these topics are too vast to give a proper summary; we will give an indication of any non-standard or particular notions.

Let $\mathcal{A}$ be a quasi-abelian category. Recall \cite[\S 1]{schneiders_quasi-abelian_1999} that this means that $\mathcal{A}$ is an additive category which has all kernels and cokernels, and strict\footnote{Recall that a morphism $f$ is called strict if the natural morphism $\operatorname{coker} \operatorname{ker}f \xrightarrow[]{ }\operatorname{ker}\operatorname{coker}f$ is an isomorphism.} epimorphisms (resp. monomorphisms) are stable under pullbacks (resp. pushouts). We recall the following properties.
\begin{defn}\cite[Definition 2.47]{kelly_homotopy_2021}
An object $P \in \mathcal{A}$ is called \emph{projective} if the functor $\operatorname{Hom}(P,-): \mathcal{A} \to \mathsf{Ab}$ (valued in abelian groups), takes strict epimorphisms to surjections.
\end{defn}
\begin{defn}\cite[Definition 2.92]{kelly_homotopy_2021}
Assume that $\mathcal{A}$ admits small coproducts. A small subcategory $\mathcal{P}$ of objects in $\mathcal{A}$ is called \emph{generating} if for each object $M \in \mathcal{A}$ there exists a small collection $\{P_i\}_{i \in \mathcal{I}}$ of objects of $\mathcal{P}$ together with a strict epimorphism $\bigoplus_{i \in \mathcal{I}} P_i \twoheadrightarrow M$.
\end{defn}
\begin{defn}\cite[Definition 2.97]{kelly_homotopy_2021} Let $\mathcal{S}$ be a collection of morphisms in a quasi-abelian category $\mathcal{A}$ stable under composition.
\begin{enumerate}[(i)]
        \item Let $\mathcal{I}$ be a filtered category. An object $C \in \mathcal{A}$ is called \emph{$(\mathcal{I},\mathcal{S})$-tiny} if the functor $\operatorname{Hom}(C,-): \mathcal{A} \to \mathsf{Ab}$ commutes with colimits of diagrams in $\operatorname{Fun}_\mathcal{S}(\mathcal{I},\mathcal{A})$. Here $\operatorname{Fun}_\mathcal{S}(\mathcal{I},\mathcal{A}) \subseteq \operatorname{Fun}(\mathcal{I},\mathcal{A})$ denotes the sub-class of those functors which take morphisms in $\mathcal{I}$ into $\mathcal{S}$.
        \item An object $C \in \mathcal{A}$ is called \emph{$\mathcal{S}$-tiny} if the functor $\operatorname{Hom}(C,-): \mathcal{A} \to \mathsf{Ab}$ commutes with colimits of diagrams in $\operatorname{Fun}_\mathcal{S}(\mathcal{I},\mathcal{A})$, for any filtered category $\mathcal{I}$. 
    \item The category $\mathcal{A}$ is called \emph{$\mathcal{S}$-elementary} if $\mathcal{A}$ is is generated by a subcategory $\mathcal{P} \subseteq \mathcal{A}$ consisting of $\mathcal{S}$-tiny projective objects. 
\end{enumerate}
\end{defn}
In what follows we will often take $(\mathcal{I}, \mathcal{S}):= (\mathbb{N}, \mathrm{SplitMon})$, where $\mathrm{SplitMon}$ is the class of split monomorphisms in $\mathcal{A}$, or $\mathcal{S} := \mathrm{AdMon}$ to be the class of strict monomorphisms in $\mathcal{A}$, or $\mathcal{S}:= \mathrm{all}$. In each of these cases we say that $\mathcal{A}$ is \emph{$(\mathbb{N}, \mathrm{SplitMon})$-elementary}, \emph{$\mathrm{AdMon}$-elementary}, and \emph{elementary}, respectively. Of course, we have the following chain of implications:
\begin{equation*}
    \text{elementary} \implies \mathrm{AdMon}\text{-elementary} \implies (\mathbb{N}, \mathrm{SplitMon})\text{-elementary}.
\end{equation*}
The following Definition is standard.
\begin{defn}
\begin{enumerate}[(i)]
    \item Let $\kappa$ be a regular cardinal. 
    \begin{enumerate}
        \item An object $C \in \mathcal{A}$ is called \emph{$\kappa$-compact} if the functor $\operatorname{Hom}(C,-)$ commutes with $\kappa$-filtered colimits.
        \item The category $\mathcal{A}$ is called \emph{$\kappa$-compactly generated} if it has all $\kappa$-filtered colimits and it is generated under $\kappa$-filtered colimits by a small subcategory of $\kappa$-compact objects. 
    \end{enumerate}
    \item The category $\mathcal{A}$ is called \emph{locally presentable} if it is cocomplete and it is $\kappa$-compactly generated, for some regular cardinal $\kappa$.
\end{enumerate}
\end{defn}

We let $\operatorname{Ch}(\mathcal{A})$ denote the category of unbounded chain complexes valued in $\mathcal{A}$; we use cohomological indexing convention. Then Kelly \cite{kelly_homotopy_2021} has shown the following. We refer the reader to \cite[\S A.1.1]{HigherToposTheory} for the definition of simplicial model categories, and \cite[\S A.2.6]{HigherToposTheory} for combinatorial model categories.

\begin{thm}\cite[Theorem 4.65]{kelly_homotopy_2021}\label{thm:kelly}
Assume that $\mathcal{A}$ is a $(\mathbb{N},\mathrm{SplitMon})$-elementary quasi-abelian category. We write $\varphi:M \to N$ for a (chain) map in $\operatorname{Ch}(\mathcal{A})$ with components $\varphi_i$. Then there exists a stable, simplicial model structure (the \emph{projective model structure}) on $\operatorname{Ch}(\mathcal{A})$, with weak equivalences, fibrations, and cofibrations given as follows: 
\begin{itemize}
    \item[(W)] $\varphi$ is a weak equivalence if it is a quasi-isomorphism, i.e., $\operatorname{cone}(\varphi)$ is strictly exact. 
    \item[(F)] $\varphi$ is a fibration if the components $\varphi_i$ are strict epimorphisms. 
    \item[(C)] $\varphi$ is a cofibration if it has the left lifting property with respect to acyclic fibrations.\footnote{Of course, this part of the definition is redundant because in any model category, any one of the three classes is determined by the other two in this way.}  
\end{itemize}
Further, this model structure gives $\operatorname{Ch}(\mathcal{A})$ the structure of a combinatorial model category, if $\mathcal{A}$ is locally presentable.
\end{thm}
\begin{rmk}
\begin{enumerate}[(i)]
    \item With this model structure, every object of $\operatorname{Ch}(\mathcal{A})$ is fibrant.    \item By the familiar dictionary between model categories, this Theorem implies that the underlying $\infty$-category $N(\operatorname{Ch}(\mathcal{A}))[W^{-1}]$ is \emph{stable}, and additionally \emph{presentable} whenever $\mathcal{A}$ is locally presentable.
\end{enumerate}

\end{rmk}
Now suppose that the category $\mathcal{A}$ is endowed with a closed symmetric monoidal structure $(\mathcal{A},\otimes,\underline{\operatorname{Hom}})$.
\begin{defn}\cite[Definition 2.70]{kelly_homotopy_2021}
An object $F \in \mathcal{A}$ is called \emph{flat} if the endofunctor $F \otimes -$ preserves the kernels of strict morphisms. 
\end{defn}
\begin{defn}\cite[Definition 2.71]{kelly_homotopy_2021}
The category $\mathcal{A}$ is called \emph{projectively monoidal} if:
\begin{enumerate}[(i)]
    \item $\mathcal{A}$ has enough projectives;
    \item All projective objects of $\mathcal{A}$ are flat;
    \item The tensor product of any two projective objects is again projective. 
\end{enumerate}
\end{defn}
We will make use of the following Theorem due to Kelly \cite{kelly_homotopy_2021}. 
\begin{thm}\cite[Theorem 4.69]{kelly_homotopy_2021}\label{thm:monoidalmodel}
Assume that $\mathcal{A}$ is projectively monoidal and elementary. Then $\mathrm{Ch}(\mathcal{A})$, endowed with the above projective model structure, is a monoidal model category satisfying the monoid axiom. 
\end{thm}
\begin{rmk}
We do not define \emph{monoidal model categories} or explain the \emph{monoid axiom} here and instead refer the reader to \cite[\S A.3.1]{HigherToposTheory}. For our purposes, it is enough to note that Theorem \ref{thm:monoidalmodel} implies that, whenever $\mathcal{A}$ is projectively monoidal and elementary, the underlying $\infty$-category $N(\operatorname{Ch}(\mathcal{A}))[W^{-1}]$ is a \emph{presentably symmetric monoidal $\infty$-category}. This means that it is \emph{presentable} and the tensor product commutes with colimits separately in each variable. 
\end{rmk}
In the remainder of this section, we will assume that $\mathcal{A}$ is projectively monoidal, elementary, and locally presentable quasi-abelian category, so that the results of Theorems \ref{thm:kelly} and \ref{thm:monoidalmodel} apply. In particular, \cite[Theorem 4.5.3.1]{HigherAlgebra}, (see also \cite[Remark 4.5.3.2]{HigherAlgebra}), implies that there exists a functor
\begin{equation}
\mathsf{CAlg}(N(\operatorname{Ch}(\mathcal{A}))[W^{-1}]) \to \mathsf{CAlg}(\mathsf{Cat}_\infty),
\end{equation}
which in fact factors through $\mathsf{CAlg}(\mathsf{Pr}^L_{\mathsf{st}}) \to \mathsf{CAlg}(\mathsf{Cat}_\infty)$. The result of \cite[Theorem A.7(ii)]{TopologicalCyclicHomology} implies that the localization functor $N(\operatorname{Ch}(\mathcal{A})) \to N(\operatorname{Ch}(\mathcal{A}))[W^{-1}]$ is lax-symmetric monoidal, and hence induces a functor on commutative-algebra objects. Using the inclusion $\mathcal{A} \subseteq \operatorname{Ch}(\mathcal{A})$ in degree $0$, we obtain by composition a functor
\begin{equation}
     N(\mathsf{CAlg}(\mathcal{A})) \simeq  \mathsf{CAlg}(N(\mathcal{A})) \to \mathsf{CAlg}(N(\operatorname{Ch}(\mathcal{A}))[W^{-1}]) \to \mathsf{CAlg}(\mathsf{Pr}^L_{\mathsf{st}}), 
\end{equation}
where the first equivalence is due to \cite[Example 2.1.3.3]{HigherAlgebra}. If we define $\mathsf{Aff}(\mathcal{A})$ to be the (nerve of) $\mathsf{CAlg}(\mathcal{A})^\mathsf{op}$, we have now defined a prestack
\begin{equation}
    \operatorname{QCoh}: \mathsf{Aff}(\mathcal{A})^\mathsf{op} \to \mathsf{CAlg}(\mathsf{Pr}^L_{\mathsf{st}}).
\end{equation}
which sends $\operatorname{Sp}(A) \in \mathsf{Aff}(\mathcal{A})$ (corresponding to $A \in \mathsf{CAlg}(\mathcal{A})$), to the presentably symmetric monoidal $\infty$-category 
\begin{equation}
    \operatorname{QCoh}(\operatorname{Sp}(A)) = \operatorname{Mod}_A(N(\operatorname{Ch}(\mathcal{A}))[W^{-1}]),
\end{equation}
where on the right-side, $A$ is viewed as a monoid via the \emph{derived} tensor product. A morphism $f: X = \operatorname{Sp}(A) \to \operatorname{Sp}(B) = Y$ in $\mathsf{Aff}(\mathcal{A}) := \mathsf{CAlg}(\mathcal{A})^\mathsf{op}$, corresponding to a nonzero morphism $f^\#: B \to A $ in $\mathsf{CAlg}(\mathcal{A})$, is sent to the (derived) pullback functor $f^* \simeq A \otimes^\mathbf{L}_B -$, which admits a right adjoint $f_*$. By construction , the $(-)^*$ functors are compatible with composition\footnote{In the $(\infty,1)$-sense, so that there are higher homotopies witnessing $(fg)^* \simeq g^*f^*$ for composable $f,g$.}, and by naturality of \emph{passing to adjoints} the same is true for the $(-)_*$ functors\footnote{See \cite[\S5.5.3]{HigherToposTheory}; the anti-equivalence $\mathsf{Pr}^L \simeq (\mathsf{Pr}^R)^\mathsf{op}$ is obtained by \emph{passing to adjoints}.}. 
\begin{prop}\label{prop:homotopymono}
For a morphism $f: X = \operatorname{Sp}(A) \to \operatorname{Sp}(B) = Y$ in $\mathsf{Aff}(\mathcal{A})$, the following are equivalent:
\begin{enumerate}[(i)]
    \item The counit of the adjunction $f^* f_* \xrightarrow[]{} \operatorname{id}$ is a natural isomorphism.
    \item The functor $f_* : \operatorname{QCoh}(\operatorname{Sp}(A)) \to \operatorname{QCoh}(\operatorname{Sp}(B))$ is fully faithful.
    \item The natural map $A \otimes^\mathbf{L}_B A \to A$ is an isomorphism in $\operatorname{QCoh}(\operatorname{Sp}(A))$. 
\end{enumerate}
\end{prop}
\begin{proof}
Omitted, see for instance \cite[\S4.3]{BBKNonArch}. 
\end{proof}
\begin{defn}\cite[\S1.2.6]{toen_homotopical_2008}
If any of the equivalent conditions of Proposition \ref{prop:homotopymono} are satisfied, $f: X \to Y$ is called a \emph{homotopy monomorphism} (and $f^\#: B \to A$ shall be called a \emph{homotopy epimorphism}).  
\end{defn}
\begin{defn}\label{defn:homotopytransversal}
A \emph{homotopy Zariski transversal pair} $(\mathsf{A},\tau)$ is a full subcategory $\mathsf{A} \subset \mathsf{Aff}(\mathcal{A})$, and a collection of homotopy monomorphisms $\tau$ in $\mathsf{A}$, such that:
\begin{enumerate}[(i)]
    \item $\tau$ contains all isomorphisms, is stable under composition, and $\tau$ is stable under base change by arbitrary morphisms in $\mathsf{A}$.
    \item $\mathsf{A}$ is stable under fiber products and finite coproducts in $\mathsf{Aff}(\mathcal{A})$, and for any morphism $u : \operatorname{Sp}(C) \to \operatorname{Sp}(B)$ in $\mathsf{A}$, and any homotopy monomorphism $t: \operatorname{Sp}(A) \to \operatorname{Sp}(B)$ in $\tau$, the natural morphism $C \otimes^\mathbf{L}_B A \to C \otimes_B A$ is an equivalence.
\end{enumerate}
\end{defn}
The last part of Definition \ref{defn:homotopytransversal} motivates the following definition. 
\begin{defn}\label{defn:transversalmorphism}
Let $(\mathsf{A},\tau)$ be a homotopy Zariski transversal pair. A morphism $t: \operatorname{Sp}(A) \to \operatorname{Sp}(B)$ in  $\mathsf{A}$ is called \emph{transversal} if, for any morphism $u : \operatorname{Sp}(C) \to \operatorname{Sp}(B)$ in $\mathsf{A}$, the natural morphism $C \otimes^\mathbf{L}_B A \to C \otimes_B A$ is an equivalence.
\end{defn}
If $(\mathsf{A},\tau)$ is a homotopy Zariski transversal pair, and $t, u$ are as in Definition \ref{defn:transversalmorphism}, then one has a natural isomorphism of functors
\begin{equation}\label{eq:basechange}
t^* u_* \xrightarrow[]{\sim} u^\prime_* t^{\prime,*}.
\end{equation}
where $t^\prime, u^\prime$ are given as in the Cartesian and homotopy Cartesian square
\begin{equation}
\begin{tikzcd}
	{\operatorname{Sp}(C \otimes_B A)} & {\operatorname{Sp}(A)} \\
	{\operatorname{Sp}(C)} & {\operatorname{Sp}(B)}
	\arrow["u^\prime",from=1-1, to=1-2]
	\arrow["t^\prime", from=1-1, to=2-1]
	\arrow["u"', from=2-1, to=2-2]
	\arrow["t", from=1-2, to=2-2]
\end{tikzcd}
\end{equation}
Indeed, this is nothing but the equivalence
\begin{equation}
    A \otimes_B^\mathbf{L} M \xrightarrow[]{\sim} (C \otimes_B^\mathbf{L} A)  \otimes_C^\mathbf{L} M \xrightarrow[]{\sim}  (C \otimes_B A)  \otimes_C^\mathbf{L} M
\end{equation}
for $M \in \operatorname{QCoh}(\operatorname{Sp}(C))$. 
In the rest of this section we shall fix a homotopy Zariski transversal pair $(\mathsf{A}, \tau)$.
\begin{defn}\label{def:covers} \cite[Prop 3.20]{BBKNonArch}
We call a finite collection of morphisms 
\begin{equation}
\{f_i : U_i = \operatorname{Sp}(A_i) \to \operatorname{Sp}(A) = X \}_{i \in \mathcal{I}}   
\end{equation}in $\mathsf{A}$ a \emph{cover} if:
\begin{equation}\label{eq:covercondition}
\text{For any }M \in \operatorname{QCoh}(\operatorname{Sp}(A)),\text{ whenever }f_i^*M \simeq 0\text{ for all }i \in \mathcal{I},\text{ then }M \simeq 0. 
\end{equation}
\end{defn}
\begin{prop}
The covers given in Definition \ref{def:covers} define a Grothendieck (pre)topology on $\mathsf{A}$, in other words:
\begin{enumerate}[(i)]
    \item Isomorphisms are covers, 
    \item If $\{U_i \to X\}_{i}$ and $\{U_{ij} \to U_i\}_{j}$ are covers then $\{U_{ij} \to X\}_{i,j}$ is a cover, 
    \item Covers are stable under base change by arbitrary morphisms in $\mathsf{A}$. 
\end{enumerate}
\end{prop}
\begin{proof}
This was already shown in \cite{BBKNonArch}. The assertion (i) is clear. Verifying the condition \eqref{eq:covercondition} in (ii) follows from the fact that pullbacks compose. Verifying \eqref{eq:covercondition} in (iii) follows from the base change isomorphism \eqref{eq:basechange}.
\end{proof}
\begin{defn}
The category $\mathsf{A}$ with the covers given in \ref{def:covers} shall be called the \emph{$\tau$-homotopy Zariski site} and denoted $\mathsf{A}_{Zar}^\tau$. 
\end{defn}
Let $\{t_i: U_i \to X\}_{i \in \mathcal{I}}$ be a covering family in $\mathsf{A}_{Zar}^\tau$. By the property of limits in $\infty$-categories, there is a morphism
\begin{equation}
    \phi: \operatorname{QCoh}(X) \to \varprojlim_{I \subseteq \mathcal{I}} \operatorname{QCoh}(U_I) 
\end{equation}
where $I$ runs over the poset of finite nonempty subsets $I = \{i_1,\dots, i_n\} \subseteq \mathcal{I}$, and $U_I := U_{i_1} \times_X \dots \times_X U_{i_n}$.\footnote{We implicitly chose some ordering on $\mathcal{I}$ and always view $I \subseteq \mathcal{I}$ with the induced ordering.} On objects, this sends $M \mapsto (t_I^*M)_I$ together with the isomorphisms induced by $u^* t^*_I = (t_I u)^* \simeq t_J^*$ for every morphism $u: U_J \to U_I$ induced by $J \supseteq I$. It is formal that $\phi$ has a right adjoint given by $(M_I)_I \mapsto \varprojlim_{I \subseteq \mathcal{I}} t_{I,*} M_I$. 
\begin{prop}\label{prop:descentqcoh}
With notations as above. The natural morphism 
\begin{equation}\label{eq:descentqcoh}
    \phi: \operatorname{QCoh}(X) \to \varprojlim_{I \subseteq \mathcal{I}} \operatorname{QCoh}(U_I) 
\end{equation}
is an equivalence. 
\end{prop}
\begin{proof}
This is the same as \cite[Proposition 10.5]{CondensedMathematics}. For the reader's convenience we reproduce the argument here. To verify fully faithfulness we check that the unit map
\begin{equation}\label{eq:counitcovering}
    M \mapsto \operatorname{lim}_{I \subseteq \mathcal{I}} t_{I,*}t_I^*(M)
\end{equation}
is an isomorphism. Since we are working with stable $\infty$-categories then $t_{U_i}^*$ commutes with finite limits. By the condition \eqref{eq:covercondition}, it suffices to check that \eqref{eq:counitcovering} is an isomorphism after applying $t^*_{U_i}$.  By base-change, we will then be in the same situation, except with $X = U_i$, and one of the covering maps is $\operatorname{id}:X \to X$. But then it is obvious that \eqref{eq:counitcovering} is an isomorphism.

For essential surjectivity, say that $(M_I)_I$ belongs to the right side of \eqref{eq:descentqcoh} and set $M := \operatorname{lim}_{I \subseteq \mathcal{I}} t_{I,*}M_I$, we must check that the counit morphism $t_J^*(M) \to M_J$ is an isomorphism for each nonempty subset $J \subseteq \mathcal{I}$. Again $t_J^*$ commutes with this finite limit, and then $\operatorname{lim}_{I 
 \subseteq \mathcal{I}}t_J^*t_{I,*}M_I \xrightarrow[]{\sim} M_J$, by base-change. 
\end{proof}
Continuing in the above setup, with $\{t_{U_i} : U_i \to X\}_{i \in \mathcal{I}}$ a covering family in $\mathsf{A}^\tau_{Zar}$, we let $\mathfrak{U}_\bullet \to X$ denote the \v{C}ech nerve of $\bigsqcup_{i \in \mathcal{I}} U_i \to X$. In an entirely similar way to the above, there is a morphism
\begin{equation}
    \psi: \operatorname{QCoh}(X) \to \varprojlim_{[m] \in \Delta} \operatorname{QCoh}(\mathfrak{U}_m),
\end{equation}
which, on objects sends $M \mapsto (t_m^*M)_{[m] \in \Delta}$ together with the isomorphisms induced by $u^*t^*_m = (t_mu)^* \simeq t_n^*$ for every simplicial morphism $\mathfrak{U}_n \to \mathfrak{U}_m$. Again, it is formal that $\psi$ has a right adjoint given on objects by $(M_m)_{[m] \in \Delta} \mapsto \varprojlim_{[m] \in \Delta} t_{m,*} M_m$.
\begin{prop}\label{prop:cechdescent}
With notations as above. The natural morphism
\begin{equation}\label{eq:cechdescent}
    \psi: \operatorname{QCoh}(X) \to \varprojlim_{[m] \in \Delta} \operatorname{QCoh}(\mathfrak{U}_m)
\end{equation}
is an equivalence. 
\end{prop}
\begin{proof}
By using Mathew's theory of descendable algebras, this follows from Proposition \ref{prop:descentqcoh} above. Let us use notations as in Proposition \ref{prop:descentqcoh}, and suppose that $X = \operatorname{Sp}(A)$ and $U_i = \operatorname{Sp}(A_i)$, so that $U_I = \operatorname{Sp}(A_I)$ where $A_I = A_{i_1} \otimes_A \dots \otimes_A A_{i_n}$ for $I = \{i_1,\dots,i_n\}$. Looking at the unit object in the equivalence \eqref{eq:descentqcoh}, we see that there are equivalences
\begin{equation}
    A \simeq \varprojlim_{I \subseteq \mathcal{I}} A_I \simeq \varprojlim_{I = (i_1,\dots,i_n) \subseteq \mathcal{I}} A_{i_1} \otimes^\mathbf{L}_A \dots \otimes^\mathbf{L}_A A_{i_n}, 
\end{equation} 
where for the second equivalence we used Definition \ref{defn:homotopytransversal}(ii). This implies that $A \in \operatorname{QCoh}(X)$ belongs to the \emph{thick tensor-ideal} of $(\operatorname{QCoh}(X), \otimes^\mathbf{L}_A)$ generated by $\prod_{i \in \mathcal{I}} A_i$, c.f. \cite[Definition 3.17]{MathewGalois}, and therefore $A \to \prod_{i \in \mathcal{I}} A_i := B$ is \emph{descendable} \cite[Definition 3.18]{MathewGalois}.
Therefore, \cite[Proposition 3.22]{MathewGalois} guarantees that
\begin{equation}
    \operatorname{Mod}_A(N(\operatorname{Ch}(\mathcal{A}))[W^{-1}]) \to \varprojlim_{[m] \in \Delta} \operatorname{Mod}_{B^{\otimes_A^\mathbf{L} (m+1)}}(N(\operatorname{Ch}(\mathcal{A}))[W^{-1}]) 
\end{equation}
is an equivalence. However, we note by Definition \ref{defn:homotopytransversal}(ii) again that $B^{\otimes_A^\mathbf{L} (m+1)} \simeq B^{\otimes_A (m+1)}$, so that \eqref{eq:cechdescent} is indeed an equivalence. 
\end{proof}
\begin{cor}\label{cor:QCohsheaf}
$\operatorname{QCoh}: \mathsf{A}^\mathsf{op} \to \mathsf{CAlg}(\mathsf{Pr}^L_\mathsf{st})$ is a sheaf on the site $\mathsf{A}^\tau_{Zar}$.
\end{cor}
\begin{proof}
Noting that the Grothendieck topology is finitary, this follows from Proposition \ref{prop:cechdescent} above, using \cite[Proposition A.3.3.1]{SpectralAlgebraicGeometry}. 
\end{proof}

From now on, for $X = \operatorname{Sp}(A) \in \mathsf{A}$ we will write the monoidal structure on $\operatorname{QCoh}(X)$ as $\otimes_X = \otimes^\mathbf{L}_A$; that is, we suppress the $\mathbf{L}$ when the subscript is a ``space".

Let $f: X = \operatorname{Sp}(A) \to \operatorname{Sp}(B) =Y $ be a morphism in $\mathsf{A}$. Since $f^*$ is symmetric monoidal and left adjoint to $f_*$, we obtain a morphism of bifunctors $f^*(f_* \otimes_Y \operatorname{id}) \cong f^*f_* \otimes_X f^* \to \operatorname{id} \otimes_X f^*$ which induces by adjunction a morphism
\begin{equation}\label{eq:projectionaffinoid}
f_* \otimes_Y \operatorname{id} \to f_*(\operatorname{id} \otimes_X f^*).
\end{equation}
\begin{lem}\label{lem:projectionaffinoid}
The morphism \eqref{eq:projectionaffinoid} is an isomorphism; in other words $f$ satisfies the projection formula.
\end{lem}
\begin{proof}
By the pointwise criterion for natural equivalence (which we may use implicitly in the rest of this article), it suffices to show that, for all $M \in \operatorname{QCoh}(X)$, $N \in \operatorname{QCoh}(Y)$, the natural morphism $M \otimes^\mathbf{L}_B N \to M \otimes^\mathbf{L}_A f^* N$ is an equivalence. This amounts to the associativity property $M \otimes^\mathbf{L}_B N \xrightarrow{\sim} M \otimes^\mathbf{L}_A A \otimes^\mathbf{L}_B N$ of $\otimes^\mathbf{L}$.
\end{proof}
\begin{lem}\label{lem:rightadjoint}
\begin{enumerate}[(i)]
    \item Let $X \in \mathsf{A}$. For fixed $M \in \operatorname{QCoh}(X)$, the functor $M \otimes_X-$ admits a right adjoint. 
    \item Let $f: X = \operatorname{Sp}(A) \to \operatorname{Sp}(B) = Y$ be a morphism in $\mathsf{A}$. The functor $f_*$ admits a right adjoint.
\end{enumerate}
\end{lem}
\begin{proof}
(i): This is immediate since $\operatorname{QCoh}(X)$ is \emph{presentably} symmetric monoidal.
(ii): Since all categories involved are presentable, it is sufficient (c.f. \cite[Corollary 5.5.2.9]{HigherToposTheory}) to show that $f_*$ commutes with colimits. We note that $f_*$ is the forgetful functor in the monadic adjunction $f^* \dashv f_*$. In any monadic adjunction, the forgetful functor preserves all colimits which the monad preserves \cite[Corollary 4.2.3.5]{HigherAlgebra}. In this case the monad is $f_*f^* = A \otimes_B^\mathbf{L} -$, which preserves colimits.
\end{proof}
For a morphism $f: X = \operatorname{Sp}(A) \to \operatorname{Sp}(B) = Y$ in $\mathsf{A}$ we will write $f^!$ for the right adjoint to $f_*$. It can be computed on objects as $f^! \simeq R \underline{\operatorname{Hom}}_B(A,-)$. 
\section{Quasicoherent sheaves}\label{sec:quasicoherent}
Let $K$ be a complete field extension of $\mathbb{Q}_p$. Let $(\mathsf{IndBan}_K, \widehat{\otimes}_K, \underline{\operatorname{Hom}}_K)$ be the $\operatorname{Ind}$-category of $K$-Banach spaces.
\begin{prop}
$(\mathsf{IndBan}_K, \widehat{\otimes}_K, \underline{\operatorname{Hom}}_K)$ is a projectively monoidal, elementary and locally presentable quasi-abelian category.
\end{prop}
\begin{proof}
See \cite[Proposition 2.1.16, Proposition 2.1.19]{schneiders_quasi-abelian_1999}. In order to guarantee that $\mathsf{IndBan}_K$ is locally presentable we may fix a sufficiently large cutoff cardinal $\kappa$ and only consider $\kappa$-small Banach spaces. 
\end{proof}
Therefore, the formalism of \S\ref{sec:derivedquasiabelian} applies with $\mathcal{A} = \mathsf{IndBan}_K$. As our model for $p$-adic analytic geometry we will use the dagger analytic spaces of \cite{GrosseKlonneRigid}, which we also direct the reader to. They are locally G-ringed spaces modelled on dagger algebras, which are defined as quotients of the Washnitzer algebra $W_n:= K \langle T_1,\dots,T_n\rangle^\dagger$ of overconvergent power series, i.e., those with radius of convergence $>1$, rather than the Tate algebra $T_n := K \langle T_1,\dots,T_n\rangle$ as in Tate's rigid geometry. The Washnitzer algebra may be presented as an $\operatorname{Ind}$-Banach space via the formula $K \langle T_1,\dots,T_n\rangle^\dagger \cong \varinjlim_{\varrho > 1} K \langle T_1/\varrho,\dots,T_n/\varrho\rangle$. In this way all dagger algebras acquires the canonical structure of an $\operatorname{Ind}$-Banach space, independently of the choice of presentation. 

Many of the usual good properties of rigid spaces carry over to the dagger setting, including Tate acyclicity. There is \cite[Theorem 2.19]{GrosseKlonneRigid} a ``completion" functor from the category of dagger analytic spaces, to the category of Tate's rigid analytic spaces, which identifies the underlying G-topological spaces, and the local rings of the structure sheaves. For us, an admissible open subset of a dagger analytic variety, will always mean an admissible open subset in the strong G-topology. As in rigid geometry, we are constrained to convergence radii from $\sqrt{|K^\times|} := |K^\times| \otimes_{\mathbb{Z}} \mathbb{Q}$. We consider the category of dagger affinoid spaces $\mathsf{Afnd}^\dagger$ as a full subcategory of $\mathsf{Aff}(\mathsf{IndBan}_K)$.

Bambozzi and Ben-Bassat \cite[Theorem 5.7]{BambozziDaggerBanach} have proven the following; note that theirs is stated in terms of Berkovich spectra, but the same argument works in our setting, where $\operatorname{Sp}(\cdot)$ shall always denote the maximal ideal spectrum.

\begin{thm}\cite[Theorem 5.7]{BambozziDaggerBanach}\label{lem:affinoidhomotop}
Let $A \to B$ be any nonzero morphism of dagger affinoid $K$-algebras, $U = \operatorname{Sp}(B_U) \hookrightarrow \operatorname{Sp}(B)$ any affinoid subdomain, so that $B \to B_U$ is a dagger affinoid localisation. Then the map
\begin{equation}
    A \widehat{\otimes}^\mathbf{L}_B B_U \to A \widehat{\otimes}_B B_U,
\end{equation}
is an isomorphism in $\operatorname{QCoh}(\operatorname{Sp}(A))$, in particular $U \to \operatorname{Sp}(B)$ is a homotopy monomorphism.
\end{thm}

Let $\tau$ be the collection of all affinoid subdomain inclusions in $\mathsf{Afnd}^\dagger$. Then $\tau$ contains all isomorphisms, is stable under composition, and stable under base change by arbitrary morphisms in $\mathsf{Afnd}^\dagger$. Thus Theorem \ref{lem:affinoidhomotop} gives:
\begin{cor}\label{cor:transversalpair}
$(\mathsf{Afnd}^\dagger, \tau)$ is a homotopy Zariski transversal pair.
\end{cor}
\begin{prop}\label{prop:conservative1}
Let $\{t_{U_i} : U_i = \operatorname{Sp}(A_{i}) \to \operatorname{Sp}(A) = X\}_{i = 1}^n$ be an admissible cover by of $\operatorname{Sp}(A) \in \mathsf{Afnd}^\dagger$ by affinoid subdomains, and let $M \in \operatorname{QCoh}(\operatorname{Sp}(A))$. If $t_{U_i}^*M \simeq 0$ for all $i = 1, \dots, n$, then $M \simeq 0$.
\end{prop}
\begin{proof}
Assume first that $n=2$ and $U_1 = X(1/a), U_2 = X(a/1)$ for some $a \in A$. By Tate acyclicity for dagger algebras, the alternating \v{C}ech complex associated to this covering yields a strictly exact sequence in $\mathsf{IndBan}_K$:
\begin{equation}
    0 \to A \xrightarrow[]{i} A\langle a/1 \rangle^\dagger \oplus A \langle 1/a \rangle^\dagger \xrightarrow[]{d} A \langle a/1, 1/a \rangle^\dagger \to 0.
\end{equation}
We now note that the map $d$ admits a bounded section $s$. More precisely, the map $s$ extends by continuity the map $A[X,Y] \to A \langle X \rangle^\dagger /(X-a) \oplus A \langle Y \rangle^\dagger/(1-aY)$ given by $s(X^iY^j) = X^{i-j}$ for $i \ge j$ and $s(X^iY^j) = Y^{j-i}$ for $i < j$. 
As the sequence is strictly exact, $i$ is the (categorical) kernel of $d$, and so  $i$ admits a retraction $r$ in $\mathsf{IndBan}_K$, c.f. \cite[Remark 7.4]{buhler_exact_2010}. Now let $M \in \operatorname{QCoh}(\operatorname{Sp}(A))$. By this discussion, the canonical morphism $M \to (A\langle a/1 \rangle^\dagger \oplus A \langle 1/a \rangle^\dagger) \widehat{\otimes}^\mathbf{L}_A M$ admits a retraction in $\operatorname{QCoh}(\operatorname{Sp}(K))$. Hence, if $t_{U_1}^*M \simeq 0$ and $t_{U_2}^*M \simeq 0$ then $M \simeq 0$ in $\operatorname{QCoh}(\operatorname{Sp}(K))$. But, the forgetful functor to $\operatorname{QCoh}(\operatorname{Sp}(K))$ is conservative, so $M \simeq 0$, proving the Proposition in this case.

In a similar way to \cite[Theorem 4.2.2]{FresnelVanDerPut}, we now deduce the claim when $U_1,\dots,U_n$ is a standard covering of $X$, i.e., $U_i = X(f_1 \cdots \hat{f}_i \cdots f_n / f_i)$ for some $f_1, \dots, f_n \in A$ generating the unit ideal; in this case we write $\{U_1,\dots,U_n\} = St(f_1,\dots,f_n)$. The case $n=1$ is trivial. Assume that we have proved the claim for standard covers of size $n$ and consider $St(f_1,\dots,f_{n+1})$. Note that, for any $ 1 \le j \le n+1$, we have: 
\begin{equation}\label{eq:starcondition}
\begin{aligned}[c]
\text{If }|f_{j}(x)| \ge |f_i(x)|\text{ for all } x \in X \text{ and all } 1 \le i \le n,\text{ then }\\ \{U_1,\dots,\hat{U}_j,\dots,U_n\} = St(f_1,\dots,\hat{f_j},\dots, f_n),
\end{aligned}
\end{equation} 
and in that case we would be done by inductive hypothesis. As $f_1,\dots,f_{n+1}$ generate the unit ideal, we can find $\varrho \in K^\times$ such that $|\varrho| < \inf_{x \in X} (\max_{i=1,\dots,n+1} |f_i(x)|)$. Consider 
\begin{equation}
\begin{aligned}
X_1 := \{ x \in X : |f_{n+1}(x) / \varrho | \le 1 \} && \text{ and } && X_2 := \{ x \in X : |f_{n+1}(x) / \varrho | \ge 1 \},
\end{aligned}
\end{equation}
and let $u_1: X_1 \to X$, $u_2:X_2 \to X$ be the inclusions. By \eqref{eq:starcondition}, $U_1 \cap X_1, \dots, U_n\cap X_1$ is a standard covering of $X_1$ of size $n$, hence by inductive hypothesis $u_1^* M \simeq 0$; by the case $n=2$ above, it suffices to show that $u_2^*M  \simeq 0$, i.e., we may assume to begin with that $|f_{n+1}(x)/\varrho| \ge 1$ for all $x \in X$. Continuing this process, we see that we may assume $|f_i(x)/\varrho| \ge 1$ for all $i = 1,\dots, n+1$, to begin with. In this case $f_1/f_2 \in A$, and we set 
\begin{equation}
\begin{aligned}
X_1 = \{ x \in X : |(f_1/f_2)(x)| \le 1 \} && \text{ and } && X_2 = \{ x \in X : |(f_2/f_1)(x)| \le 1 \}, 
\end{aligned}
\end{equation}
and by \eqref{eq:starcondition}, we see that $\{U_2 \cap X_1, \dots,U_{n+1} \cap X_{1} \} = St(f_2,\dots,f_{n+1})$, and $\{U_1 \cap X_2, U_3 \cap X_2, \dots,U_{n+1} \cap X_{2} \} = St(f_1,f_3,\dots,f_{n+1})$. Hence by inductive hypothesis, $u_1^*M \simeq 0 $ and $u_2^*M \simeq 0$, and we are done by the $n=2$ step.  
For a general affinoid covering, we can use the Gerritzen-Grauert theorem to refine to a finite rational covering, and then \cite[Lemma 4.2.4]{FresnelVanDerPut} to refine to a finite standard covering, whence the claim follows from the above.
\end{proof}
Let $\mathsf{Afnd}^\dagger_w$ denote the ``big" weak topology on $\mathsf{Afnd}^\dagger$, i.e., $\{U_i \to X\}_{i \in \mathcal{I}}$ is a cover if it is an admissible cover by affinoid subdomains.  By Corollary \ref{cor:transversalpair}, Corollary \ref{cor:QCohsheaf}, and Proposition \ref{prop:conservative1}, we obtain the following. 
\begin{thm}\label{thm:Cechdescentweak}
$\operatorname{QCoh}$ is a $\mathsf{CAlg}(\mathsf{Pr}_{\mathsf{st}}^L)$-valued sheaf on $\mathsf{Afnd}^\dagger_w$. 
\end{thm}
Let $\mathsf{Rig}^\dagger_s$ be the big site of all dagger analytic varieties, equipped with the strong topology. Since $\mathsf{Afnd}^\dagger_w$ is a basis for $\mathsf{Rig}^\dagger_s$,  \cite[Proposition A.3.11(ii)]{mann_p-adic_2022} implies that the right Kan extension of functors along $\mathsf{Afnd}^\dagger \to \mathsf{Rig}^\dagger$ extends $\operatorname{QCoh}$ to a $\mathsf{CAlg}(\mathsf{Pr}_{st}^L)$-valued sheaf on $\mathsf{Rig}^\dagger_s$\footnote{To avoid set-theoretic issues we restrict to dagger analytic varieties with affinoid coverings of cardinality less than some strongly inaccessible limit cardinal $\kappa$.}, given objectwise by
\begin{equation}\label{eq:QCohOtimes}
\begin{aligned}
\operatorname{QCoh}(X) := \lim_{(\operatorname{Sp}(A) \to X)} \operatorname{QCoh}(\operatorname{Sp}(A)),
\end{aligned}
\end{equation}
where the limit runs over the full subcategory of $\mathsf{Rig}^\dagger_{/X}$ spanned by affinoid dagger spaces mapping to $X$.

Since $\operatorname{QCoh}$ is a $\mathsf{CAlg}(\mathsf{Pr}^L_{\mathsf{st}})$-valued sheaf on $\mathsf{Rig}^\dagger_s$, every map $f: X \to Y$ of dagger analytic varieties induces a symmetric monoidal pullback functor $f^* : \operatorname{QCoh}(Y) \to \operatorname{QCoh}(X)$. As this is a morphism in $\mathsf{Pr}^L_{\mathsf{st}}$, by the property of presentable $\infty$-categories \cite[Corollary 5.5.2.9]{HigherToposTheory} it has an (abstractly defined) right adjoint 
\begin{equation}
f^* : \operatorname{QCoh}(Y) \leftrightarrows \operatorname{QCoh}(X) : f_*,  
\end{equation}
which we call pushforwards. We write $\widehat{\otimes}_X$ for the monoidal structure of $\operatorname{QCoh}(X)$. By \emph{passing to adjoints}, the formation of $(-)_*$ is compatible with composition. Hence, for any Cartesian square 
\begin{equation}
\begin{tikzcd}
	X^\prime & Y^\prime \\
	X & Y
	\arrow["f"', from=2-1, to=2-2]
	\arrow["{f^\prime}", from=1-1, to=1-2]
	\arrow["{t^\prime}"', from=1-1, to=2-1]
	\arrow["t", from=1-2, to=2-2]
	\arrow["\lrcorner"{anchor=center, pos=0.125}, draw=none, from=1-1, to=2-2]
\end{tikzcd}
\end{equation}
we obtain a base change morphism
\begin{equation}
    t^* f_* \xrightarrow[]{} f_*^\prime t^{\prime, *},
\end{equation}
which, to be precise, is the morphism obtained by adjunction from the composite
\begin{equation}
    f_* \to f_* t^\prime_* t^{\prime,*} \to t_* f^\prime_* t^{\prime,*}, 
\end{equation}
where the first morphism is obtained from the unit of $t^{\prime,*} \dashv t^\prime_*$. Since $f^*$ is symmetric monoidal and left adjoint to $f_*$ we obtain the composite
\begin{equation}
   f^*(f_*\widehat{\otimes}_Y \operatorname{id}) \xrightarrow[]{\sim} f^*f_*\widehat{\otimes}_X f^* \to \operatorname{id} \widehat{\otimes}_X f^*,
\end{equation}
where the second morphism is obtained from the counit of $f^* \dashv f_*$, which by adjunction gives a morphism
\begin{equation}
    f_* \widehat{\otimes}_Y \operatorname{id} \to f_*(\operatorname{id} \widehat{\otimes}_X f^*).
\end{equation}
By functoriality of $f_*$, for any (small) diagram $(M_i)_{i \in \mathcal{I}}$ in $\operatorname{QCoh}(X)$ one has a natural morphism 
\begin{equation}\label{eq:colimI}
    \varinjlim_{i \in \mathcal{I}} f_*M_i \to f_* \varinjlim_{i \in \mathcal{I}}  M_i,
\end{equation}
therefore, the statement of Proposition \ref{prop:pushforward} below makes sense.
\begin{prop}\label{prop:pushforward}
Suppose that $f: X \to Y$ is a qcqs morphism of dagger analytic varieties. Then:
\begin{enumerate}[(i)]
    \item \cite[Lemma 2.4.16]{mann_p-adic_2022} The functor $f_*$ commutes with restrictions, i.e., if $t: W \hookrightarrow Y$ is an admissible open subset, in the induced Cartesian diagram
\begin{equation}
\begin{tikzcd}
	W^\prime & W \\
	X & Y
	\arrow["{f^\prime}", from=1-1, to=1-2]
	\arrow["{t^\prime}"', hook, from=1-1, to=2-1]
	\arrow["f"', from=2-1, to=2-2]
	\arrow["t", hook, from=1-2, to=2-2]
	\arrow["\lrcorner"{anchor=center, pos=0.125}, draw=none, from=1-1, to=2-2]
\end{tikzcd}
\end{equation}
the induced base change morphism 
\begin{equation}\label{eq:inftycatbasechange}
t^* f_*    \rightarrow f^\prime_* t^{\prime,*} 
\end{equation} 
is an equivalence.
\item $f_*$ admits a right adjoint (which we denote by $f^!$).
\item $f_*$ satisfies the projection formula: the natural morphism
\begin{equation}\label{eq:projectionformula}
    f_* \widehat{\otimes}_Y \operatorname{id} \to f_*(\operatorname{id} \widehat{\otimes}_X f^*)
\end{equation}
is an equivalence.
\end{enumerate}
\end{prop}
\begin{proof}(i): This is essentially the same as \cite[Lemma 2.4.16]{mann_p-adic_2022}; we reproduce the argument here for the reader's convenience. 

First assume that $X$ and $Y$ are both dagger affinoid varieties. One chooses a cover $\{W_i \to W\}_{i \in \mathscr{I}}$ of $W$ by affinoid subdomains of $Y$. Let $\mathcal{I}$ be the family of finite nonempty subsets of $\mathscr{I}$ and for each $I \in \mathcal{I}$ set $W_I := \bigcap_{i \in I} W_i$ and $W_I^\prime := f^{-1}(W_I)$. Let $t_I: W_I \to Y$ and $t_I^\prime: W_I^\prime \to X$ be the inclusions and let $f_I^\prime: W_I^\prime \to W_I$. Each $W_I^\prime$ and $W_I$ is affinoid. Hence, by Proposition \ref{prop:descentqcoh} and \eqref{eq:basechange}, one has $ t^*f_* \simeq \varprojlim_{I \in \mathcal{I}} t_I^* f_* \simeq \varprojlim_{I \in \mathcal{I}} f_{I,*}^\prime t_I^{\prime,*} \simeq f^\prime_* t^{\prime,*}$, proving the Lemma in this case.

Next we assume that $Y$ is affinoid and $X$ is an admissible open subspace of a dagger affinoid variety $Z$. Let $\{U_i \to X\}_{i \in \mathscr{I}}$ be a (finite, since $f$ is qcqs) covering of $X$ by affinoid subdomains of $Z$ and let $\mathcal{I}$ be the family of all finite nonempty subsets $I \subseteq \mathscr{I}$. For $I \in \mathcal{I}$ we again set $U_I := \bigcap_{i \in I}U_i$. According to Proposition \ref{prop:descentqcoh}, then $f_* \xrightarrow[]{\sim} \varprojlim_{I \in \mathcal{I}}f_{I,*}$ where $f_I : U_I \to X$ is the restriction. Since $t^*$ commutes with finite limits, we deduces base-change in this case from the preceding.

Now assume that $Y$ is an admissible open subset of a dagger affinoid $Z$ and $X$ is arbitary. One chooses a cover $\{V_i \to Y\}_{i \in \mathscr{I}}$ of  $Y$ by affinoid subdomains of $Z$. Let $\mathcal{I}$ be the family of finite nonempty subsets of $\mathscr{I}$ and for each $I \in \mathcal{I}$ set $V_I := \bigcap_{i \in I} V_i$ and $V_I^\prime := f^{-1}(V_I)$ and let $f_I: V_I^\prime \to V_I$. Each $V_I$ is affinoid and hence by the preceding each $f_I$ commutes with restrictions. Therefore the collection $(f_{I,*})_{I \in \mathcal{I}}$ preserves Cartesian sections. Since the covering of $Y$ was chosen arbitrarily we conclude that $f_*$ commutes with restrictions. 

For arbitary $X$, $Y$, cover $\{V_i \to Y\}_{i \in \mathcal{I}}$ of $Y$ by affinoid subspaces. Let $\mathcal{I}$ be the family of finite subsets of $\mathscr{I}$ and for each $I \in \mathcal{I}$ set $V_I := \bigcap_{i \in I} V_i$ and $V_I^\prime := f^{-1}(V_I)$ and let $f_I: V_I^\prime \to V_I$. Each $V_I$ is an admissible open subspace of a dagger affinoid space and hence by the preceding, each $f_I$ commutes with restrictions. Hence, as before, $f_*$ commutes with restrictions. 

(ii): Since the categories are presentable, it suffices by Lurie's Adjoint Functor Theorem \cite[Corollary 5.5.2.9]{HigherToposTheory} to show that $f_*$ commutes with colimits. Let $(M_i)_{i \in \mathcal{I}}$ be a (small) diagram in $\operatorname{QCoh}(X)$. We need to show that \eqref{eq:colimI}
is an equivalence. Take affinoid coverings $\mathfrak{V}$ of $Y$ and $\mathfrak{U}$ of $X$ such that $\mathfrak{U}$ refines $f^{-1}\mathfrak{V}$. For each $U \in \mathfrak{U}$, let $t_{U}: U \to X$ be the inclusion. The collection $\{t_U^*: U \in \mathfrak{U}\}$ is conservative and so it suffices to check that \eqref{eq:colimI} is an equivalence after applying each $t^*_U$. By the commutation with restrictions, and since $t_U^*$ commutes with colimits, we reduce to when $X$ and $Y$ are both dagger affinoids, which is Lemma \ref{lem:rightadjoint}(ii). 

(iii): Using the same notations as in the proof of (ii), it suffices to check that \eqref{eq:projectionformula} is an equivalence after applying each $t_U^*$. By the commutation with restrictions, and using that each $t_U^*$ is symmetric-monoidal, one reduces to the case when $X$ and $Y$ are both affinoids, which is Lemma \ref{lem:projectionaffinoid}.
\end{proof}
\subsection{Transversal morphisms}
It is often necessary to have a base-change theorem for more general morphisms than the inclusions of admissible open subsets. The purpose of this subsection is to formulate and prove a base-change result for a certain class of morphisms which we call \emph{transversal}. We then show that every smooth morphism of dagger-analytic varieties belongs to this class. We recall the definition of a \emph{transversal morphism} in $\mathsf{Afnd}^\dagger$ from \S\ref{sec:derivedquasiabelian}.
\begin{defn}\label{defn:daggertransversal}
A morphism $t: \operatorname{Sp}(A) \to \operatorname{Sp}(B)$ in $\mathsf{Afnd}^\dagger$ is called \emph{transversal}, or \emph{induced by a transversal morphism of algebras}, if either of the following equivalent conditions are satisfied:
\begin{enumerate}[(i)]
    \item For all maps $u : \operatorname{Sp}(C) \to \operatorname{Sp}(A)$ in $\mathsf{Afnd}^\dagger$, the natural morphism $B \widehat{\otimes}^\mathbf{L}_A C \to B\widehat{\otimes}_A C$ is an equivalence.
    \item For all maps $u: \operatorname{Sp}(C) \to \operatorname{Sp}(A)$, the natural morphism $t^*u_* \to u^\prime_* t^{\prime,*}$ is an equivalence. Here $t^\prime, u^\prime$ are given as in the Cartesian square:
\begin{equation}
\begin{tikzcd}
	{\operatorname{Sp}(C \widehat{\otimes}_B A)} & {\operatorname{Sp}(A)} \\
	{\operatorname{Sp}(C)} & {\operatorname{Sp}(B)}
	\arrow["u^\prime",from=1-1, to=1-2]
	\arrow["t^\prime", from=1-1, to=2-1]
	\arrow["u"', from=2-1, to=2-2]
	\arrow["t", from=1-2, to=2-2]
\end{tikzcd}
    \end{equation}
\end{enumerate}
\end{defn}
\begin{example}\label{example:steadinessaffinoid}
\begin{enumerate}[(i)]
    \item If $t: \operatorname{Sp}(A_U) \to \operatorname{Sp}(A)$ is the embedding of an affinoid subdomain, then $t$ is transversal \cite[Theorem 5.7] {BambozziDaggerBanach}. (We used this fact in the previous subsection). 
    \item If $\operatorname{Sp}(A)$ is any $K$-dagger affinoid, then the structure morphism $t: \operatorname{Sp}(A) \to \operatorname{Sp}(K)$ is transversal. In fact, $A$ is actually flat with respect to the completed tensor product. This is because every Banach space of countable type is isomorphic (as a $K$-Banach space) to $c_0(\mathbb{N})$, c.f. \cite[Proposition 10.4]{schneider_nonarchimedean_2002}, and every dagger affinoid algebra can be expressed as an $\mathbb{N}$-indexed colimit of such objects (via the familiar presentation of the Washnitzer algebra as a colimit of Tate algebras).
\end{enumerate}
\end{example}
\begin{lem}\label{lem:afndtransversalbasechangecompose}
\begin{enumerate}[(i)]
    \item The class of transversal morphisms in $\mathsf{Afnd}^\dagger$ is stable under composition and base change. 
    \item Let $f: X \to Y$ and $g: Y \to Z$ be composable morphisms in $\mathsf{Afnd}^\dagger$ such that both $gf$ and the diagonal $\Delta_g: Y \to Y \times_Z Y$ are transversal. Then $f$ is transversal. 
\end{enumerate}
\end{lem}
\begin{proof}
(i): Immediate from Definition \ref{defn:daggertransversal}. (ii): $f$ can be factored as $X \to X \times_Z Y \to Y$. The first morphism is the base change of $\Delta_g$; the second is the base-change of $gf$. Therefore, the result follows from (i).
\end{proof}
\begin{lem}\cite[Lemma 2.4.18]{mann_p-adic_2022}\label{lem:transversallocal}
Let $f: \operatorname{Sp}(B) \to \operatorname{Sp}(A)$ be a morphism of dagger affinoid algebras. Suppose that there exists an admissible covering $\{t_{i}: \operatorname{Sp}(B_{i}) \to \operatorname{Sp}(B)\}_{i \in \mathcal{I}}$ of $\operatorname{Sp}(B)$ by affinoid subdomains, such that each $ft_{i} : \operatorname{Sp}(B_{i}) \to \operatorname{Sp}(A)$ is transversal. Then $f$ is transversal.
\end{lem}
\begin{proof} This follows from Definition \ref{defn:daggertransversal}(ii) and descent, noting that all relevant pushforwards commute with restrictions. 
\end{proof}
Inspired by \cite[\S 2.4]{mann_p-adic_2022}, we now globalise the notion of transversality to morphisms between dagger-analytic varieties.
\begin{defn}\cite[Definition 2.4.17]{mann_p-adic_2022}\label{defn:transversalglobal}
A morphism $f:X \to Y$ of dagger-analytic varieties is called \emph{transversal} if for all affinoid subdomains $U \subset Y$ and all affinoid subdomains $V \subset f^{-1}(U)$, the morphism $V \to U$ is induced by a transversal morphism of algebras.
\end{defn}
\begin{lem}\cite[Lemma 2.4.19]{mann_p-adic_2022}\label{lem:transversalsourceandtarget}
Let $f: X \to Y$ be a morphism of dagger-analytic varieties. The following are equivalent:
\begin{enumerate}[(i)]
    \item $f$ is transversal.
    \item There exists an admissible covering $\{U_i\}_{i \in \mathcal{I}}$ of $Y$ by affinoid subdomains, such that each $f^{-1}(U_i) \to U_i$ is transversal.
    \item There exists an admissible covering $\{U_i\}_{i \in \mathcal{I}}$ of $Y$ by affinoid subdomains, and for each $i$ an admissible covering $\{V_{ij}\}_{j \in \mathcal{J}(i)}$ of $f^{-1}(U_i)$ by affinoid subdomains, such that each $V_{ij} \to U_i$ is induced by a transversal morphism of algebras.
\end{enumerate}
\end{lem}
\begin{proof}
(i) $\implies$ (ii) $\implies$ (iii) is clear. 

(iii) $\implies$ (ii): Let $i \in \mathcal{I}$ and let $U \subseteq U_i$ and $V \subseteq f^{-1}(U_i)$ be affinoid subdomains. We know that each $V_{ij} \to U_i$ is induced by a transversal morphism of algebras. Refine $\{V_{ij} \cap V\}_j$ further to an admissible affinoid cover $\{W_{ik}\}_k$ of $V$ by affinoids. Then each $W_{ik} \to U$ is induced by a transversal morphism of algebras. Then Lemma \ref{lem:transversallocal} above implies that $V \to U$ is induced by a transversal morphism of algebras.

(ii) $\implies$ (i): Let $U \subseteq Y$ and $V \subseteq f^{-1}(U)$ be affinoid subdomains. We know that each $f^{-1}(U_i) \to U_i$ is transversal. Refine $\{U \cap U_i\}_{i \in \mathcal{I}}$ to an admissible covering $\{W_k\}_k$ of $U$ by affinoid subdomains, and refine $\{f^{-1}W_k \cap V\}_k$ to an admissible covering $\{Z_l\}_l$ of $V$ by affinoid subdomains. Then each $Z_l \to W_k$ (indices chosen appropriately) is induced by a transversal morphism of algebras, hence each $Z_l \to U$ is induced by a transversal morphism of algebras, so by Lemma \ref{lem:transversallocal}, $V \to U$ induced by a transversal morphism of algebras. Hence $f$ is transversal. 
\end{proof}
\begin{lem}
\begin{enumerate}[(i)]
    \item The class of transversal morphisms in $\mathsf{Rig}^\dagger$ is closed under base-change and composition.
    \item Let $f: X \to Y$ and $g: Y \to Z$ be composable morphisms in $\mathsf{Rig}^\dagger$ such that both $gf$ and the diagonal $\Delta_g: Y \to Y \times_Z Y$ are transversal. Then $f$ is transversal.
\end{enumerate}
\end{lem}
\begin{proof}
(i): The claim about composition follows from Lemma \ref{lem:transversalsourceandtarget}(iii). Indeed, let $X \xrightarrow[]{f} Y \xrightarrow[]{g} Z$ be composable morphisms which are both transversal. Choose an admissible cover $\{Z_k\}_k$ of $Z$ by affinoid subdomains, choose an admissible affinoid cover $\{Y_j\}_j$ of $Y$ refining $\{g^{-1}(Z_k)\}_k$ and an admissible affinoid cover $\{X_i\}_i$ of $X$ refining $\{f^{-1}(Y_j)\}_j$. Then each $X_i \to Y_j$ and each $Y_j \to Z_k$ is induced by a transversal morphism of algebras, hence so is $X_i \to Z_k$ (all indices chosen appropriately). Hence by Lemma \ref{lem:transversalsourceandtarget}(iii), $gf$ is transversal.

The claim about base-change also follows from Lemma \ref{lem:transversalsourceandtarget}(iii). Indeed, let $f: X \to Y$ be transversal and let $g: Y^\prime \to Y$ be arbitrary. Choose an admissible affinoid cover $\{Y_j\}_j$ of $Y$ and admissible affinoid covers $\{Y^\prime_{j^\prime}\}_{j^\prime}$ of $Y^\prime$ and $\{X_i\}_i$ of $X$ which refine $\{g^{-1}(Y_j)\}_j$ and $\{f^{-1}(Y_j)\}_j$, respectively. Then each $X_i \times_{Y_j} Y^\prime_{j^\prime}$ is affinoid and the morphism $X_i \times_{Y_j} Y^\prime_{j^\prime} \to Y^\prime_{j^\prime}$ (indices chosen appropriately) is induced by a transveral morphism of algebras, by Lemma \ref{lem:afndtransversalbasechangecompose}(i). Hence we conclude by Lemma \ref{lem:transversalsourceandtarget}(iii).

(ii): Follows from (i) using the usual graph trick.  
\end{proof}
\begin{prop}[Transversal base-change]\label{prop:transversalbasechange}
Let $f: X \to Y$ and $t: Y^\prime \to Y$ be morphisms of dagger-analytic varieties such that $f$ is qcqs and $t$ is transversal. Then the base-change morphism 
\begin{equation}
    t^* f_* \to f_*^\prime t^{\prime, *},
\end{equation}
is an equivalence. Here $t^\prime, f^\prime$ are determined as in the Cartesian square:
\begin{equation}
\begin{tikzcd}[cramped]
	{X^\prime} & {Y^\prime} \\
	X & Y
	\arrow["{f^\prime}", from=1-1, to=1-2]
	\arrow["{t^\prime}"', from=1-1, to=2-1]
	\arrow["t", from=1-2, to=2-2]
	\arrow["f"', from=2-1, to=2-2]
\end{tikzcd}
\end{equation}
\end{prop}
\begin{proof}
This is proved exactly as in \cite[Proposition 2.4.21]{mann_p-adic_2022}. However, we reproduce the proof here to convince the reader that it is true in or context. 

Suppose first that $Y$ and $Y^\prime$ are both affinoid and $X$ is an admissible open subset of an affinoid space $Z$. Choose an admissible affinoid covering $\{X_i\}_{i=1}^n$ of $X$ and let $\mathcal{I}$ be the family of finite nonempty subsets $I \subseteq \{1, \dots, n\}$ and for each $I \in \mathcal{I}$ let $X_I$ be the corresponding intersection and $f_I$ be the restriction. Each $X_I$ is affine. By descent, $f_* \simeq \varprojlim_{I} f_{I,*}$. Hence, we may conclude in this case using Lemma \ref{lem:afndtransversalbasechangecompose} and the fact that $t^*$ commutes with finite limits.

Now the case when $Y$ and $Y^\prime$ are both affinoid (and hence $X$ is qcqs), is handled in a similar way, because we now know the base-change result in the case when $X$ is a quasiaffinoid. 

In general, by the commutation of $f_*$ with restrictions (Proposition \ref{prop:pushforward}) and descent one reduces to the case when $Y$ and $Y^\prime$ are both affinoid, hence $X$ is qcqs, proving the claim.  
\end{proof}
The following Proposition provides plenty of examples of transversal morphisms. 
\begin{prop}[Smooth base-change]\label{prop:smoothbasechange}
Let $g: X \to Y$ be a smooth morphism of dagger-analytic varieties. Then $g$ is transversal. 
\end{prop}
\begin{proof}
By Lemma \ref{lem:transversalsourceandtarget}(iii) above, we may work locally on the source and target. Thus, we may assume that $g$ factors as $g = ph$ where $p: \mathbb{D}^n_K \times Y \to Y$ is the projection off some (closed) polydisk and $h$ is \'etale. Let us deal with these cases separately.

\emph{$p$ is transversal}: This follows from Example \ref{example:steadinessaffinoid}(ii) and the fact that transversal morphisms are stable under base-change. 

\emph{$h$ is transversal}: We will show that every \'etale morphism of dagger-analytic varieties is transversal. By Lemma \ref{lem:transversalsourceandtarget}(iii) we reduce to the case when $h$ is a morphism between affinoids. By using the structure theorem for \'etale morphisms\footnote{This is stated for rigid-analytic varieties, but it works just as well for dagger-analytic varieties.} \cite[Proposition 3.1.4]{vanderPutetale}, and Lemma \ref{lem:transversalsourceandtarget}(iii) again, one reduces to the case when $h$ is induced by a finite \'etale morphism $h^\#: B \to A$ of dagger affinoid algebras. Working locally using Lemma \ref{lem:transversalsourceandtarget}(iii) again, we may assume that $A$ is finite free as a $B$-module. In this case it is clear that Definition \ref{defn:daggertransversal}(i) is satisfied.
\end{proof}
\subsection{Coherent sheaves}\label{sec:coherentsheaves}
In this section we investigate a notion of ``coherent sheaves".

First, let $\mathcal{A}$ be an elementary quasi-abelian category as in \S\ref{sec:derivedquasiabelian}. We recall that the stable $\infty$-category $N(\operatorname{Ch}(\mathcal{A}))[W^{-1}]$ is equipped with a canonical $t$-structure, the \emph{left $t$-structure} of \cite[\S 1.2]{schneiders_quasi-abelian_1999}. The heart of this $t$-structure is the left abelian envelope $LH(\mathcal{A})$. This contains $\mathcal{A}$ as a full subcategory via the functor $I: \mathcal{A} \to LH(\mathcal{A})$, c.f. \cite[\S 1.2.3]{schneiders_quasi-abelian_1999}.

Let $A$ be a $K$-dagger affinoid algebra. By \cite[Lemma 2.3.96]{DAnG} there is a $t$-structure on $\operatorname{QCoh}(\operatorname{Sp}(A))$ such that
\begin{equation}
   \operatorname{QCoh}(\operatorname{Sp}(A))^\heartsuit \cong  \operatorname{Mod}_{I(A)} LH(\mathsf{IndBan}_K). 
\end{equation}
Due to \cite[Proposition 2.6]{bode_six_2021} and the flatness of $A$ with respect to the completed tensor product\footnote{In order to apply the result of \cite[Proposition 2.6]{bode_six_2021}, it suffices to show that the natural morphism $H^0(A \widehat{\otimes}_K^\mathbf{L} V) \to I(A \widehat{\otimes}_K V)$ is an isomorphism, for any $V \in \mathsf{IndBan}_K$. This follows from the flatness of $K \to A$.} (c.f. Example \ref{example:steadinessaffinoid}(ii)), there is an equivalence of categories
\begin{equation}
   LH(\operatorname{Mod}_A(\mathsf{IndBan}_K)) \xrightarrow[]{\sim} \operatorname{Mod}_{I(A)} LH(\mathsf{IndBan}_K).
\end{equation}
Hence, we obtain cohomology functors 
    \begin{equation}
    \begin{aligned}
             H^i :   \operatorname{QCoh}(\operatorname{Sp}(A)) \to  LH(\mathsf{Mod}_A(\mathsf{IndBan}_K)), && i \in \mathbb{Z}.
    \end{aligned}
    \end{equation}
There is a fully-faithful functor
\begin{equation}
    \{\text{abstract finitely-generated $A$-modules}\}  \hookrightarrow \mathsf{Mod}_A(\mathsf{IndBan}_K),
\end{equation}
which equips a finitely-generated $A$-module with its canonical bornology \cite[Corollary 6.1.6]{BambozziThesis}. Putting this all together, we obtain a fully-faithful functor from the category of abstract finitely-generated $A$-modules, to $\operatorname{QCoh}(\operatorname{Sp}(A))^\heartsuit$.
In particular the following definition makes sense. 
\begin{defn}
Let $\operatorname{Sp}(A) \in \mathsf{Afnd}^\dagger$. We define $\operatorname{Coh}^-(\operatorname{Sp}(A)) \subseteq \operatorname{QCoh}(\operatorname{Sp}(A))$ as the full sub $\infty$-category spanned by objects $M$ such that $H^n(M)$ is a finitely generated $A$-module for all $n \in \mathbb{Z}$, and $H^n(M) = 0$ for $n \gg 0$. 
\end{defn}
\begin{lem}\label{lem:Coh-pullback}
\begin{enumerate}[(i)]
    \item Let $f: \operatorname{Sp}(A) \to \operatorname{Sp}(B)$ be a morphism in $\mathsf{Afnd}^\dagger$ and let $M \in \operatorname{Coh}^-(\operatorname{Sp}(B))$. Then $f^*M \in \operatorname{Coh}^-(\operatorname{Sp}(A))$. 
    \item For any $M, N \in \operatorname{Coh}^-(\operatorname{Sp}(A))$, $M \widehat{\otimes}^\mathbf{L}_A N$ belongs to $\operatorname{Coh}^-(\operatorname{Sp}(A))$.
\end{enumerate}
\end{lem}
\begin{proof}
(i): Since $M$ is cohomologically bounded above, there is\footnote{In this section, $\operatorname{Tor}$ always refers to the derived functors of the \emph{completed} tensor product.} a convergent spectral sequence 
\begin{equation}\label{eq:spectral1}
\begin{aligned}
E_2^{pq} : \operatorname{Tor}^B_p(A, H^qM) \Rightarrow H^{q-p}(f^*M).
\end{aligned}
\end{equation}
Since $B$ is Noetherian, the finitely-generated $B$-module $H^qM$ has a resolution by finite-free $B$-modules and hence $\operatorname{Tor}^B_p(A, H^qM)$ is a finitely-generated $A$-module, for all $p, q \in \mathbb{Z}$ (and zero for $q \gg 0$). Hence, by convergence of \eqref{eq:spectral1} we see that  $H^{n}(f^*M)$ is a finitely-generated $A$-module for all $n \in \mathbb{Z}$ (and $H^n(f^*M) = 0$ for $n \gg 0$).

(ii): Since $M, N$ are cohomologically bounded above, there is a convergent spectral sequence 
\begin{equation}\label{eq:spectral2}
\begin{aligned}
    E_2^{pq}: \bigoplus_{q^\prime + q^{\prime \prime} = q}\operatorname{Tor}^A_p(H^{q^\prime}M, H^{q^{\prime \prime}}N) \Rightarrow H^{q-p}(M \widehat{\otimes}_A^\mathbf{L} N).
\end{aligned}
\end{equation}
Since $A$ is Noetherian, each $H^{q^\prime}N, H^{q^{\prime \prime}}M$ has a resolution by finite-free $A$-modules and hence $\operatorname{Tor}^A_p(H^{q^\prime}M, H^{q^{\prime \prime}}N)$ is a finitely-generated $A$-module, for all $p, q^\prime, q^{\prime \prime} \in \mathbb{Z}$ (and zero for $ q^\prime \gg 0$ or $q^{\prime \prime} \gg 0$). Hence, by convergence of \eqref{eq:spectral2} we see that $H^n(M \widehat{\otimes}_A^\mathbf{L} N)$ is a finitely-generated $A$-module for all $n \in \mathbb{Z}$ (and zero for $n \gg 0$). 
\end{proof}
By Lemma \ref{lem:Coh-pullback} we obtain a sub-prestack $\operatorname{Coh}^- : (\mathsf{Afnd}^\dagger)^\mathsf{op} \to \mathsf{CAlg}(\mathsf{Cat}_\infty)$ of $\operatorname{QCoh}$, valued in stable symmetric monoidal $\infty$-categories. Therefore the statement of the following Proposition makes sense. 
\begin{prop}\label{prop:Cohdescent1}
Let $X \in \mathsf{Afnd}^\dagger$ and let $\{U_i \to X\}_{i \in \mathcal{I}}$ be a (finite) covering of $X$ by affinoid subdomains. Then the natural morphism
\begin{equation}\label{eq:Cohdescent1}
    \operatorname{Coh}^-(X) \to \varprojlim_{I \subseteq \mathcal{I}} \operatorname{Coh}^-(U_I)
\end{equation}
is an equivalence. Here the limit on the right runs over the finite nonempty subsets $I \subseteq \mathcal{I}$. 
\end{prop}
\begin{proof}
Let $t_I: U_I \to X$ be the inclusions and let us write $X = \operatorname{Sp}(A), U_I = \operatorname{Sp}(A_I)$. Let $(M_I)_{I \subseteq \mathcal{I}}$ be an object of the right-side of \eqref{eq:Cohdescent1}. Because of Lemma \ref{lem:Coh-pullback} and Proposition \ref{prop:descentqcoh}, the only thing to show is that $\varprojlim_{I \subseteq \mathcal{I}} t_{I,*} M_I$ belongs to $\operatorname{Coh}^-(X)$. We consider the hypercohomology spectral sequence
\begin{equation}\label{eq:spectral3}
\begin{aligned}
    E_2^{pq} : H^p\varprojlim_{I \subseteq \mathcal{I}} H^q(t_{I,*} M_I) \Rightarrow H^{p+q} \varprojlim_{I \subseteq \mathcal{I}} t_{I,*} M_I.
\end{aligned}
\end{equation}
For $J \supseteq I$ there are natural equivalences $A_J \widehat{\otimes}^\mathbf{L}_{A_I} M_I \xrightarrow[]{\sim} M_J$, (because $(M_I)_I$ is a coCartesian section). The spectral sequence 
\begin{equation}
 E_2^{pq}: \operatorname{Tor}_p^{A_I}(A_J, H^q M_I) \Rightarrow H^{q-p}(A_J \widehat{\otimes}^\mathbf{L}_{A_I} M_I)
\end{equation}
collapses because finitely-generated modules are transversal to affinoid localizations, and gives a natural isomorphism $A_J \widehat{\otimes}_{A_I} H^qM_I \cong  H^q(t_{J,*} M_J)$. Therefore the dagger-analytic version of Tate acyclicity \cite[Proposition 2.6]{GrosseKlonneRigid} implies that $H^p\varprojlim_{I \subseteq \mathcal{I}} H^q(t_{I,*} M_I) = 0$, whenever $p>0$. Therefore, the spectral sequence \eqref{eq:spectral3} collapses and the edge morphism gives an isomorphism
\begin{equation}
\begin{aligned}
    H^{n} \varprojlim_{I \subseteq \mathcal{I}} t_{I,*} M_I &\cong H^0\varprojlim_{I \subseteq \mathcal{I}} H^n(t_{I,*} M_I) \\
    &= \operatorname{eq} \bigg( \prod_i H^n M_i \rightrightarrows \prod_{i,j}  H^n M_{ij} \bigg),
\end{aligned}
\end{equation}
which is a finitely generated $A$-module, by the theorem of descent for coherent sheaves on dagger spaces \cite[Theorem 2.16]{GrosseKlonneRigid}. 
\end{proof}
\begin{prop}\label{prop:cohdescent2}
Let $X \in \mathsf{Afnd}^\dagger$ and let $\{U_i \to X\}_{i \in \mathcal{I}}$ be a (finite) covering of $X$ by affinoid subdomains. Let $\mathfrak{U}_\bullet \to X$ be the \v{C}ech nerve of $\bigsqcup_{i \in \mathcal{I}} U_i \to X$. Then, the canonical morphism
\begin{equation}
    \operatorname{Coh}^-(X) \to \varprojlim_{[m] \in \Delta} \operatorname{Coh}^-(\mathfrak{U}_m)
\end{equation}
is an equivalence.
\end{prop}
\begin{proof}
This is the same, \emph{mutandis mutatis}, as the proof of Proposition \ref{prop:Cohdescent1}, using Proposition \ref{prop:cechdescent} in place of Proposition \ref{prop:descentqcoh}. 
\end{proof}
\begin{cor}
\begin{enumerate}[(i)]
    \item $\operatorname{Coh}^-:(\mathsf{Afnd}^\dagger)^\mathsf{op} \to \mathsf{CAlg}(\mathsf{Cat}_\infty)$ is a sheaf on the site $\mathsf{Afnd}^{\dagger}_w$. 
    \item Right Kan extension of functors along $(\mathsf{Afnd}^\dagger)^\mathsf{op} \to (\mathsf{Rig}^\dagger)^\mathsf{op}$ makes $\operatorname{Coh}^-$ into a sheaf on the site $\mathsf{Rig}^\dagger_s$ of all dagger-analytic varieties equipped with the strong topology. 
\end{enumerate}
\end{cor}
\begin{proof}
(i): Noting that the weak topology is finitary, this follows from \cite[Proposition A.3.14]{mann_p-adic_2022}, using Proposition \ref{prop:cohdescent2} above.

(ii): Follows from (i) using \cite[A.3.11(ii)]{mann_p-adic_2022}. \end{proof}
\section{Pushforward with compact supports}\label{sec:compactsupports}
In analytic geometry, there are many interesting morphisms which are not quasi-compact. The pushforwards along such morphisms will not, in general, have the good properties of Proposition \ref{prop:pushforward} and Proposition \ref{prop:transversalbasechange}. In \S\ref{subsubsec:compactpush} below, we define a compactly supported pushforwards functor $f_!$, for any morphism $f:X \to Y$ of dagger analytic varieties, which coincides with $f_*$ for qcqs morphisms. This definition allows us to bootstrap the good properties of $f_*$ from the qcqs case, to arbitrary morphisms. In particular $f_!$ commutes with restrictions, satisfies the projection formula, and admits a right adjoint $f^!$.

In \S\ref{subsubsec:domainembeddings} below we investigate the functor $h_!$ for inclusions of admissible open subsets $h:U \to X$. In particular for $U \subset X$ such that the inclusion of the complement $k: V := X \setminus U \to X$ is a quasi-compact morphism, we obtain an ``excision" fiber sequence in Proposition \ref{prop:hlowershriek1}, and one has $h^* \simeq h^!$ for such inclusions. Unsurprisingly, many of the proofs in these sections are similar to \cite[Lecture XII]{CondensedComplexGeometry}. 

In the mostly technical \S\ref{subsubsec:compositionf_!} below we give a partial answer to the question of when the definition of $f_!$ can be simplified, and when, given morphisms $f:X \to Y$ and $g:Y \to Z$, one has $g_! \circ f_! \simeq (gf)_!$. We show that this holds whenever $g$ is quasi compact or $g$ is partially proper and $Z$ is quasi-paracompact.

\subsection{The functor $f_!$ and its properties}\label{subsubsec:compactpush}
Let $f:X \to Y$ be any morphism of dagger analytic varieties. 
\begin{defn}\label{def:relcompact}
We let $c(X/Y)$ denote the poset of all admissible open subsets $Z \subseteq X$ such that: 
\begin{itemize}
    \item[$\star$] $Z$ admits a strict neighbourhood\footnote{Recall that this means $\{X\setminus Z, Z^\prime\}$ is an admissible covering in the strong G-topology.} $Z^\prime$ such that the composite $Z^\prime \hookrightarrow X \xrightarrow{f} Y$ is qcqs.   
    \item[$\star$]  Let $i: X \setminus Z \to X$ be the inclusion. Then we require that $i$ is a quasi-compact morphism.
\end{itemize}
\end{defn}
\begin{defn}\label{defn:compactsupportsrelY}
We say that $M \in \operatorname{QCoh}(X)$ is \emph{compactly supported (relative to $Y$)} if there exists $Z \in c(X/Y)$ such that $i^*M \simeq 0$, where $i : X \setminus Z \to X$ is the inclusion.
\end{defn}
\begin{example}
 If $f$ is qcqs then $X \in c(X/Y)$ is terminal.  
\end{example}
\begin{example}\label{example:disk}
Let $f: X := \mathring{\mathbb{D}}^1_K \to \mathbb{D}^1_K =: Y$ be the inclusion of open unit disk into the closed unit disk. For an interval $I \subset \mathbb{R}$ with endpoints in $\sqrt{|K|^\times}$, let $\mathbb{A}_K(I)$ be the associated dagger analytic annulus over $K$. From this we see that the family 
\begin{equation}
    \left\{\mathring{\mathbb{D}}^1_K(\varrho) : \varrho \in \sqrt{|K|^\times}, \varrho < 1\right\},
\end{equation}
is contained in $c(X/Y)$: setting $Z = \mathring{\mathbb{D}}^1_K(\varrho)$, one can take $Z^\prime = \mathbb{D}^1_K(\varrho^{1/2})$, in the definition, and the morphism $X \setminus Z = \mathbb{A}_K([\varrho, 1)) \to X$ is quasi-compact. It is easy to see that this family is cofinal in $c(X/Y)$. Hence, one has that $N \in \operatorname{QCoh}(X)$ is compactly supported, if and only if there exists $\varrho < 1$ such that $i_\varrho^* N \simeq 0$, where $i_\varrho: \mathbb{A}_K([\varrho, 1)) \to \mathring{\mathbb{D}}^1_K$ is the inclusion.
\end{example}
\begin{lem}\label{lem:cXYbasechange}
\begin{enumerate}[(i)]
    \item Suppose that we are given a Cartesian square 
\begin{equation}
\begin{tikzcd}
	{X^\prime} & {Y^\prime} \\
	X & Y
	\arrow["{f^\prime}", from=1-1, to=1-2]
	\arrow["f", from=2-1, to=2-2]
	\arrow["{g^\prime}", from=1-1, to=2-1]
	\arrow["g", from=1-2, to=2-2]
	\arrow["\lrcorner"{anchor=center, pos=0.125}, draw=none, from=1-1, to=2-2]
\end{tikzcd}
\end{equation}
of dagger analytic varieties over $K$. Then $g^{\prime,-1}$ maps $c(X/Y)$ to $c(X^\prime/Y^\prime)$.
\item Let $f: X \to Y$ and $g: W \to X$ be morphisms of dagger-analytic varieties. If $Z \in c(X/Y)$ and $C \in c(W/X)$ then $g^{-1}(Z) \cap C \in c(W/Y)$.
\end{enumerate}

\end{lem}
\begin{proof}
Both statements follow straightforwardly from the Definition \ref{defn:compactsupportsrelY}. 
\end{proof}
\begin{lem}\label{lem:counitiso}
Let $Z \in c(X/Y)$ and let $Z^\prime$ be given as in Definition \ref{def:relcompact}. Let $i:X \setminus Z \to X$ and $j: Z^\prime \to X$ the inclusions. 
Then: 
\begin{enumerate}[(i)]
    \item  The counit $i^*i_* \to \operatorname{id}$ is an equivalence.
    \item Suppose that $i^*M \simeq 0$. Then:
    \begin{enumerate}
        \item For any $N \in \operatorname{QCoh}(X)$ one has $i^*(M \widehat{\otimes}_XN) \simeq 0$;
        \item The unit morphism $M \to j_*j^*M$ is an equivalence. 
    \end{enumerate}
\end{enumerate}
\end{lem}
\begin{proof}
(i): Since $i: X \setminus Z \to X$ is a quasi-compact morphism, this follows from Proposition \ref{prop:pushforward}. (ii)(a): This is immediate since $i^*$ is symmetric monoidal. (ii)(b): By using the admissible covering $\{X \setminus Z, Z^\prime\}$, this follows from the fact that $\operatorname{QCoh}$ is a sheaf in the strong topology.
\end{proof}
Again, let $Z \in c(X/Y)$ and let $i: X \setminus Z \to X$ and $h: Z \to X$ be the inclusions. We may define, for $M \in \operatorname{QCoh}(X)$, the functor
\begin{equation}\label{eq:defishriekstar}
h_!^*M := \operatorname{Fib}(M \to i_*i^*M),
\end{equation}
and we define $\Gamma_Z\operatorname{QCoh}(X)$ to be the full subcategory of $\operatorname{QCoh}(X)$ on those objects $M$ with $i^*M \simeq 0$.
\begin{lem}\label{lem:ishriekstarcolimpreserve}
The functor $h_!^*$ is colimit preserving, and its essential image is contained in $\Gamma_Z\operatorname{QCoh}(X)$.
\end{lem}
\begin{proof}
Since $i$ is quasi-compact, the functor $i_*$ commutes with colimits by Proposition \ref{prop:pushforward}, and hence so does $h_!^*$, by the property of stable $\infty$-categories. The equivalence $i^*h^!_* \simeq 0$ follows from the property of stable $\infty$-categories and base-change.
\end{proof}

\begin{defn}
Let $f: X \to Y$ be a morphism of dagger analytic varieties. We define the \emph{pushforward with compact supports (relative to $Y$)}, denoted $f_!$, as the left Kan extension of $f_*$ from the full subcategory of compactly supported (relative to $Y$) objects, (c.f. Definition \ref{defn:compactsupportsrelY}).
\end{defn}
\begin{defn}\label{defn:exhaustible}
Let $f: X \to Y$ be a morphism of dagger analytic varieties. We say that $f$ is \emph{exhaustible} if $f$ is quasiseparated and $\{Z : Z \in c(X/Y)\}$ is an admissible covering of $X$.
\end{defn}
\begin{lem}\label{lem:exhaustibleproperties}
\begin{enumerate}[(i)]
    \item The class of exhaustible morphisms is stable under base-change and composition.
    \item If $f,g$ are composable morphisms of dagger-analytic varieties such that $fg$ is exhaustible and $f$ is quasiseparated then $g$ is exhaustible.
\end{enumerate}
\end{lem}
\begin{proof}
(i) By Lemma \ref{lem:cXYbasechange}(i), the class of exhaustible morphisms is stable under base-change. By Lemma \ref{lem:cXYbasechange}(ii), if $f: X \to Y$ and $g : W \to X$ are exhaustible, then $\{ g^{-1}(Z) \cap C: Z \in c(X/Y), C \in c(W/Y)\}$ is a cover of $W$ by objects of $c(W/Y)$ and hence $fg$ is exhaustible.

(ii): Follows from (i) by the usual graph trick, noting that every qcqs morphism is exhaustible. 
\end{proof}
\begin{lem}\label{lem:Indsystems}\cite[Lemma 12.14]{CondensedComplexGeometry}
\begin{enumerate}[(i)]
    \item Let $M \in \operatorname{QCoh}(X)$. Then the natural morphism 
\begin{equation}
``\underrightarrow\lim"h_!^*M \to ``\underrightarrow\lim"N  
\end{equation}
is an equivalence in $\operatorname{Ind}(\operatorname{QCoh}(X))$, where former Ind-system runs over all $Z \in c(X/Y)$, and $h: Z \to X$ is the inclusion, and the latter Ind-system runs over all compactly supported (relative to $Y$) $N$ mapping to $M$.
\item Assume that $f$ is exhaustible. Then the natural morphism 
\begin{equation}\label{eq:exhaustibleiso}
    \varinjlim h_!^*M \to M
\end{equation}
is an equivalence, where the colimit on the left runs over any subset of $c(X/Y)$ which covers $X$. 
\end{enumerate}

\end{lem}
\begin{proof}
(i): We need to show that this is well-defined, and that the former is cofinal in the latter. Let $i: X \setminus Z \to X$ be the inclusion. If $N$ satisfies $i^*N \simeq 0$ for some $Z \in c(X/Y)$, then, by naturality of $\operatorname{id} \to i_*i^*$, we see that $N \to M$ factors through $h^*_! M := \operatorname{Fib}(M \to i_*i^*M)$, and this belongs to $\Gamma_Z \operatorname{QCoh}(X)$ by Lemma \ref{lem:ishriekstarcolimpreserve}.

(ii): The collection of functors $t^*$, where $t: Z \to X$ ranges over the inclusions of those $Z \in c(X/Y)$ in the covering, is conservative, and each $t^*$ is colimit-preserving. Hence, it suffices to check that \eqref{eq:exhaustibleiso} is an equivalence after applying $t^*$. But then, by base-change, the diagram defining the colimit $\varinjlim t^*h_!^*M$ has terminal object $t^*t_!^*M \simeq M$, proving the claim. 
\end{proof}
By properties of left Kan extension, there is a natural transformation $f_! \to f_*$ which restricts to the identity on compactly supported (relative to $Y$) objects. By the coend formula $f_!$ is given on objects by 
\begin{equation}
    f_! M = \varinjlim f_* N
\end{equation}
where the colimit runs over all compactly supported (relative to $Y$) $N$ with a map to $M$. By Lemma \ref{lem:Indsystems} this can also be computed as 
\begin{equation}\label{eq:defcompactsupports}
    f_! M = \varinjlim_{Z \in c(X/Y)} f_*h_!^*M.
\end{equation}
\begin{prop}\label{prop:compactlysupported}
Let $f: X \to Y$ be any morphism of dagger-analytic varieties and suppose that $M \in \operatorname{QCoh}(X)$ is compactly supported (relative to $Y$). Then:
\begin{enumerate}[(i)]
    \item For any $N \in \operatorname{QCoh}(Y)$, the natural morphism $f_* M \widehat{\otimes}_Y N \to f_*(M \widehat{\otimes}_X f^*N)$ is an equivalence. 
    \item Let $t: Y^\prime \to Y$ be any transversal morphism and let $f^\prime, t^\prime$ be determined by the Cartesian square:
    \begin{equation}\label{eq:cartesianf}
\begin{tikzcd}[cramped]
	{X^\prime} & {Y^\prime} \\
	X & Y
	\arrow["{f^\prime}", from=1-1, to=1-2]
	\arrow["{t^\prime}"', from=1-1, to=2-1]
	\arrow["\ulcorner"{anchor=center, pos=0.125}, draw=none, from=1-1, to=2-2]
	\arrow["t", from=1-2, to=2-2]
	\arrow["f", from=2-1, to=2-2]
\end{tikzcd}
    \end{equation}
Then, the natural morphism $t^*f_*M \to f^\prime_* t^{\prime,*}M $ is an equivalence. 
\end{enumerate}
\end{prop}
\begin{proof}
Let us take $Z \in c(X/Y)$ such that $i^*M \simeq 0$, where $i: X\setminus Z \to X$ is the inclusion. Let $Z^\prime$ be given as in Definition \ref{defn:compactsupportsrelY} and let $j: Z^\prime \to X$ be the inclusion. The morphism $fj: Z^\prime \to Y$ is qcqs. Hence, for (i), we may use the projection formula for $fj$ to obtain a chain of equivalences (for arbitrary $N$):
\begin{equation}
\begin{aligned}
    f_* M \widehat{\otimes}_Y N &\simeq ((fj)_*j^*M) \widehat{\otimes}_Y N\\
    &\simeq (fj)_* (j^*M \widehat{\otimes}_{Z^\prime}(fj)^* N) \\
    &\simeq f_*j_* j^*(M \widehat{\otimes}_X f^* N) \\
    &\simeq f_* (M \widehat{\otimes}_X f^* N),
\end{aligned}
\end{equation}
where we used Lemma \ref{lem:counitiso}. For (ii), let $Z^{\prime \prime}$, $t^{\prime \prime}$ and $j^\prime$ be determined by the Cartesian square:
\begin{equation}
\begin{tikzcd}[cramped]
	{Z^{\prime\prime}} & {X^\prime} \\
	{Z^\prime} & X
	\arrow["{j^\prime}", from=1-1, to=1-2]
	\arrow["{t^{\prime \prime}}"', from=1-1, to=2-1]
	\arrow["\ulcorner"{anchor=center, pos=0.125}, draw=none, from=1-1, to=2-2]
	\arrow["{t^\prime}", from=1-2, to=2-2]
	\arrow["j"', from=2-1, to=2-2]
\end{tikzcd}
\end{equation}
By transversal base-change for $fj$, one has a chain of equivalences 
\begin{equation}\label{eq:basechangeformula!}
    t^*f_*M \simeq t^* (fj)_* j^* M \simeq (f^\prime j^\prime)_* t^{\prime \prime,*} j^*M 
    \simeq f^\prime_* j^\prime_* j^{\prime,*} t^{\prime, *}M 
    \simeq f^\prime_* t^{\prime,*} M,
\end{equation}
where we used Lemma \ref{lem:counitiso} again.
\end{proof}
\begin{prop}\label{prop:lowershriekbasechange}
Let $f: X \to Y$ be any morphism of dagger analytic varieties.
\begin{enumerate}[(i)]
    \item The functor $f_!$ admits a right adjoint $f^!$.
    \item Assume that $f$ is exhaustible. There is a natural equivalence $f_! \widehat{\otimes}_Y \operatorname{id} \xrightarrow[]{\sim} f_! (\operatorname{id} \widehat{\otimes}_X f^*)$ which restricts to $f_* \widehat{\otimes}_Y \operatorname{id} \xrightarrow[]{\sim} f_*(\operatorname{id} \widehat{\otimes}_X f^*)$ whenever the first argument is compactly supported (relative to $Y$).
    \item Assume that $f$ is exhaustible. Let $t: Y^\prime \to Y$ be any transversal morphism and let $f^\prime, t^\prime$ be determined as in the Cartesian square \eqref{eq:cartesianf}. Then, there is a natural equivalence $t^*f_! \xrightarrow[]{\sim} f^\prime_! t^{\prime,*}$, which restricts to $t^*f_* \xrightarrow[]{\sim} f^\prime_* t^{\prime, *}$ on compactly supported (relative to $Y$) objects. 
\end{enumerate}
\end{prop}
\begin{proof}
(i): Since all categories involved are presentable, it suffices by \cite[Corollary 5.5.2.9]{HigherToposTheory} to show that $f_!$ is colimit-preserving. By \eqref{eq:defcompactsupports}, it suffices to show that, for each $Z \in c(X/Y)$, the functor $f_*h_!^*$ is colimit-preserving. With $Z^\prime$ and $j: Z^\prime \to X$ given as in Definition \ref{defn:compactsupportsrelY}, there is an equivalence $f_*h^*_! \simeq (fj)_* j^*h^*_!$, (where we used Lemma \ref{lem:counitiso} and Lemma \ref{lem:ishriekstarcolimpreserve}), and this is indeed colimit-preserving by Proposition \ref{prop:pushforward} and Lemma \ref{lem:ishriekstarcolimpreserve}. 

(ii): The asserted natural transformation $
f_! \widehat{\otimes}_Y \operatorname{id} \xrightarrow[]{} f_! (\operatorname{id} \widehat{\otimes}_X f^*)$ exists, by Kan extension. In order to prove that it is an equivalence let us fix $L \in \operatorname{QCoh}(Y)$ and $M \in \operatorname{QCoh}(X)$ and consider 
\begin{equation}\label{eq:ML}
    f_! M\widehat{\otimes}_Y L \xrightarrow[]{} f_! (M \widehat{\otimes}_X f^*L). 
\end{equation}
Since $f$ is exhaustible, by Lemma \ref{lem:Indsystems} we may write $M$ as a colimit $M = \varinjlim N$ of compactly supported (relative to $Y$) objects $N$. By part (i) both sides of \eqref{eq:ML} commute with colimits in $M$. Hence, we reduce to the case when $M = N$ is compactly supported, which is Proposition \ref{prop:compactlysupported}(i) above. 

(iii): Again, the asserted natural transformation $t^*f_! \xrightarrow[]{} f^\prime_! t^{\prime,*}$ exists, by Kan extension. In order to prove that it is an equivalence, we take $M \in \operatorname{QCoh}(X)$ and consider 
\begin{equation}\label{eq:Mbasechg}
t^*f_!M \xrightarrow[]{} f^\prime_! t^{\prime,*}M.
\end{equation}
Since $f$ is exhaustible, $M$ may be written as a colimit $M = \varinjlim N$ of compactly supported (relative to $Y$) objects $N$. By part (i) both sides of \eqref{eq:Mbasechg} commute with colimits in $M$. Hence, we reduce to the case when $M = N$ is compactly supported, which is Proposition \ref{prop:compactlysupported}(ii) above. 
\end{proof}
\subsection{The excision sequence}\label{subsubsec:domainembeddings}
\begin{prop}\label{prop:hlowershriek1}
Suppose that $V \subseteq X$ is an admissible open subset such that $k: V \hookrightarrow X$ is a quasi-compact morphism. Assume that $U:= X \setminus V$ is an admissible open subset and assume that the inclusion $h:U \hookrightarrow X$ is exhaustible (Definition \ref{defn:exhaustible}). Then for any $M \in \operatorname{QCoh}(X)$ there is a fiber-cofiber sequence
\begin{equation}
    h_! h^*M \to M \to k_*k^*M.
\end{equation}
\end{prop}
\begin{rmk}\label{rmk:exhaustiblesufficient}
By Corollary \ref{cor:exhaustiblecriterion} and \cite[Example 8.1.15]{HuberEtale}, the assumption in Proposition \ref{prop:hlowershriek1} that $h: U \hookrightarrow X$ is exhaustible, is automatically satisfied whenever $X$ is quasi-paracompact and separated. 
\end{rmk}
Before the proof of Proposition \ref{prop:hlowershriek1}, let us give an example of its application.
\begin{example}
Let $X := \mathbb{D}^1_K = \operatorname{Sp}(K \langle T \rangle^\dagger)$ be the (dagger) unit disk over $K$, and let $U := \mathring{\mathbb{D}}^1_K$ be the open unit disk. Set $V:= \operatorname{Sp}(K \langle T^{-1}, T\rangle^\dagger)$, and let $h: U \to X, k :V \to X$ be the inclusions. One has $U = X \setminus V$ and $k$ is a quasi-compact morphism. Let $1_U \in \operatorname{QCoh}(U)$ be the tensor unit. Since the pullback functors are symmetric monoidal, by Proposition \ref{prop:hlowershriek1}(ii), we obtain a fiber sequence
\begin{equation}
    h_!1_U \to K \langle T \rangle^\dagger \to K\langle T^{-1}, T \rangle^\dagger,
\end{equation}
in other words 
\begin{equation}
h_! 1_U \simeq \operatorname{Fib}(K \langle T \rangle^\dagger \to K\langle T^{-1}, T \rangle^\dagger) \simeq (K \langle T^{-1}, T\rangle^\dagger/ K \langle T \rangle^\dagger)[-1],
\end{equation}
in $\operatorname{QCoh}(X)$. This lives in degree $1$ only, which reflects the fact that there are no compactly supported global holomorphic functions on $\mathring{\mathbb{D}}^1_K$. 
\end{example}
\begin{lem}\label{lem:kan1blah}
With notations as in Proposition \ref{prop:hlowershriek1}. There is an equivalence of $\infty$-categories between:
\begin{enumerate}[(i)]
    \item The $\infty$-category of all compactly supported (relative to $X$) objects $L \in \operatorname{QCoh}(U)$ with a map to $h^*N$. 
    \item The $\infty$-category of all objects $F \in \operatorname{QCoh}(X)$ with a map to $N$, such that the support of $F$ is contained in $c(U/X)$.
\end{enumerate}
We note additionally that for the objects $F$ appearing in (ii), the unit morphism $F \to h_*h^*F$ is an equivalence. 
\end{lem}
\begin{proof}
The equivalence is via ``extension by zero". If $[F \to N]$ belongs to the category in (ii) then certainly $[h^*F \to h^*N]$ belongs to the category in (i). By descent, this functor from (ii) to (i) is fully-faithful, because of the support condition on $F$.  

On the other hand, if $L \in \operatorname{QCoh}(U)$ then there exists $Z \in c(U/X)$ such that the pullback $i^*L \simeq 0$, where $i: U \setminus Z \to U$ is the inclusion. By descent applied to the admissible covering $(U \setminus Z \cup V) \cup U$ of $X$, there exists $F \in \operatorname{QCoh}(X)$ whose pullback to $(U \setminus Z \cup V)$ is nullhomotopic and such that $h^*F \simeq L$. This is equipped with a map $F \to N$ obtained by gluing\footnote{Here we implicitly used descent, and the fact that the mapping space in a limit of categories, is the limit of the mapping spaces.} $h^*F \simeq L \to h^*N$ with the morphism from $0$. Therefore, the functor from (ii) to (i) constructed in the previous paragraph is essentially surjective. 

The very last sentence is also a consequence of descent.
\end{proof}
\begin{proof}[Proof of Proposition \ref{prop:hlowershriek1}]
The first task is to produce a natural transformation $h_!h^* \to \operatorname{id}$. To this end, we recall that, by the coend formula, for $M \in \operatorname{QCoh}(X)$ one has 
\begin{equation}
    h_!h^*M  \simeq \varinjlim_{L \to h^*N} h_*L, 
\end{equation}
where the colimit runs over all compactly supported (relative to $X$) objects $L \in \operatorname{QCoh}(U)$ with a map to $h^*N$. By Lemma \ref{lem:kan1blah} above this can equivalently be written as 
\begin{equation}
    \varinjlim_{F \to N} h_*h^*F \simeq  \varinjlim_{F \to N} F
\end{equation}
where the colimit runs over the category of all objects $F \in \operatorname{QCoh}(X)$ with a map to $N$ such that the support of $F$ is contained in $c(U/X)$. Then, of course, there is a natural map 
\begin{equation}
    \varinjlim_{F \to N} F \to N,
\end{equation}
which gives the desired natural transformation $h_!h^*N \to N$.

By naturality of the unit transformation $\operatorname{id} \to k_*k^*$ and the base-change of Proposition \ref{prop:lowershriekbasechange}, there is a commutative (fracture) square:
\begin{equation}
    \begin{tikzcd}[cramped]
	{h_!h^*M} & 0 \\
	M & {k_*k^*M}
	\arrow[from=1-1, to=1-2]
	\arrow[from=1-1, to=2-1]
	\arrow[from=1-2, to=2-2]
	\arrow[from=2-1, to=2-2]
\end{tikzcd}
\end{equation}
which yields a natural transformation \begin{equation}\label{eq:excisionFib}
    h_!h^*M \to \operatorname{Fib}(M \to k_*k^*M).
\end{equation}
To check that \eqref{eq:excisionFib} is an equivalence, it suffices to do so after applying $t^*$ for $t: X^\prime \to X$ ranging over an admissible affinoid covering. Since both sides commute with restrictions, this allows us to reduce the claim to the case where $X$ is affinoid. Then $V \subset X$ is a special admissible open subset in the sense of \cite{SchneiderPoints}. Then by \cite[Lemma 2.5.1]{vanderPutetale} there exists a finite affinoid cover $\{X_i\}_i$ of $X$ such that $X_i \cap V$ is a finite union of Weierstrass domains defined by invertible functions. Thus by the commutation with restrictions again, we may reduce to this case, i.e., $X$ is affinoid and $V \subset X$ is a finite union of Weierstrass domains defined by invertible functions. Let us say $V = \bigcup_{i=1}^m V_i$ and 
\begin{equation}
V_i = \{x \in X: |f_{ij}(x)| \le 1 \text{ for all }1 \le j \le n_i\},  
\end{equation}
for functions $f_{i1}, \dots, f_{in_i}$ invertible on $V_i$. For $\varrho \in \sqrt{|K^\times|}, \varrho >1$, we set
\begin{equation}
\begin{aligned}
V_i(\varrho) &:= \{x \in X: |f_{ij}(x)| \le \varrho \text{ for all }1 \le j \le n_i\}, \\ 
\mathring{V}_i(\varrho) &:= \{x \in X: |f_{ij}(x)| < \varrho \text{ for all }1 \le j \le n_i\},
\end{aligned}
\end{equation}
note that 
\begin{equation}
X \setminus \mathring{V}_i(\varrho) = \{x \in X: |f_{ij}(x)| \ge \varrho \text{ for some }1 \le j \le n_i\},
\end{equation}
is a finite union of rational domains. We define 
\begin{equation}
\begin{aligned}
    V := \bigcup_{i=1}^m V_i,  &&  V(\varrho) := \bigcup_{i=1}^m V_i(\varrho), && \mathring{V}(\varrho) := \bigcup_{i=1}^m \mathring{V}_i(\varrho).
\end{aligned}
\end{equation}
By \cite[Lemma 2.3.1]{vanderPutetale}, $\{V(\varrho)\}_{\varrho > 1}$ is a cofinal system of strict neighbourhoods of $V$ in $X$. Recall that $U := X \setminus V$ and $h:U \to X$ is the inclusion. Note that $X \setminus V(\varrho) \in c(U/X)$ since $V(\varrho) \to X$ is a quasi-compact morphism and $X \setminus V(\varrho)$ has the quasi-compact strict neighbourhood $X \setminus \mathring{V}(\varrho^\frac{1}{2})$ in $U$. Let $k_\varrho: V(\varrho) \to X$ be the inclusions. The sheaf property (plus the fact that $X \setminus V(\varrho) \in c(U/X)$), implies that for any $\varrho > 1$, there is a natural equivalence
\begin{equation}
\begin{aligned}
    \operatorname{Fib}(M \to k_{\varrho,*}k_\varrho^*M) &\xrightarrow[]{\sim} h_*h^*\operatorname{Fib}(M \to k_{\varrho,*}k_\varrho^*M)\\
    &= h_!h^*\operatorname{Fib}(M \to k_{\varrho,*}k_\varrho^*M).
\end{aligned}
\end{equation}
By Lemma \ref{lem:cofinalnbhd} below, the natural morphism
\begin{equation}\label{eq:overconvergentclaim}
\varinjlim_{\varrho > 1}k_{\varrho, *}k_{\varrho}^* \to k_* k^*,
\end{equation}
is an equivalence. Hence using that $h_!$ commutes with colimits we obtain a natural equivalence
\begin{equation}
\begin{aligned}\label{eq:counit1}
    \operatorname{Fib}(M \to k_*k^*M) &\simeq h_! h^*\operatorname{Fib}(M \to k_*k^*M) \\
    &= h_!h^*M,
\end{aligned}
\end{equation}
where in the last line we used that $h^*k_* \simeq 0$, by base-change. Hence we have proven the Proposition modulo the following Lemma.
\end{proof}
\begin{lem}\label{lem:cofinalnbhd}
Suppose that $X$ is affinoid and $V \subset X$ is a finite union of Weierstrass domains, defined by invertible functions. Let $\{V(\varrho)\}_\varrho$ be any cofinal system of quasi-compact strict neighbourhoods of $V$ in $X$ and let $k_\varrho: V(\varrho) \to X$ be the inclusions. Then the natural morphism
\begin{equation}
\varinjlim_{\varrho > 1}k_{\varrho, *}k_{\varrho}^* \to k_*k^*,  
\end{equation}
is an equivalence. 
\end{lem}
\begin{proof}
In this situation the proof of \cite[Lemma 2.5.2]{vanderPutetale} shows that there is a finite (rational) cover $\{X_i\}$ of $X$ such that $W_i := X_i \cap V$ is rational for all $i$. Let $g: W_i \to X_i$ and $k_{\varrho,i}: V(\varrho) \cap X_i \to X_i$ be the inclusions. By the commutation with restrictions again we are reduced to proving that the natural morphism
\begin{equation}\label{eq:rationalvarrho}
    \varinjlim_{\varrho > 1}k_{\varrho,i, *}k_{\varrho,i}^* \to  g_* g^*,
\end{equation}
is an equivalence. Note that $\{V(\varrho) \cap X_i\}_\varrho$ is a cofinal system of strict neighbourhoods of $W_i$ in $X_i$. Say $X_i= \operatorname{Sp}(A)$ and $W_i = X_i\left(\frac{f_1, \dots, f_m}{f_0}\right) = \operatorname{Sp}(A_{W_i})$ for $f_0,\dots, f_m \in A$ without common zero. For $\varrho \in \sqrt{|K^\times|}, \varrho > 1$, let $W_i(\varrho) := X_i\left(\frac{f_1, \dots, f_m}{\varrho f_0}\right) = \operatorname{Sp}(A_{W_i(\varrho)})$, and let $g_\varrho: W_i(\varrho) \to X$ be the inclusions. Then by \cite[Lemma 2.3.1]{vanderPutetale}, $\{W_i(\varrho)\}_{\varrho > 1}$ is cofinal in the system of strict neighbourhoods of $W_i$ in $X_i$. Hence by cofinality we are reduced to proving that the natural morphism
\begin{equation}
    \varinjlim_{\varrho > 1} g_{\varrho,*}g_\varrho^* \to g_*g^*.
\end{equation}
is a equivalence. By the property of the Washnitzer algebra one has 
\begin{equation}\label{eq:washnitzercolimit}
    K\langle T_1,\dots,T_m \rangle^\dagger \cong \varinjlim_{\varrho > 1}K \langle T_1/\varrho,\dots, T_m/\varrho\rangle^\dagger, 
\end{equation}
indeed, the right side of \eqref{eq:washnitzercolimit} is a repeated colimit of Tate algebras, which returns the Washnitzer algebra again. Therefore, we obtain
\begin{equation}
\begin{aligned}
A_{W_i} &= \frac{A\langle T_1,\dots,T_m\rangle^\dagger}{(f_0T_1-f_1,\dots,f_0T_m-f_m)}  \\ &\simeq \varinjlim_{\varrho > 1}\frac{A\langle T_1/\varrho,\dots,T_m/\varrho\rangle^\dagger}{(f_0T_1-f_1,\dots,f_0T_m-f_m)} = \varinjlim_{\varrho > 1}A_{W_i(\varrho)},
\end{aligned}
\end{equation}
in $\operatorname{QCoh}(X)$, which proves \eqref{eq:rationalvarrho}.
\end{proof}
\begin{prop}\label{prop:hlowershriekadjunction}
Let $V \subset X$ be an admissible open subset such that the inclusion $k:V \hookrightarrow X$ is a quasi-compact morphism and let $h: U:= X \setminus V \to X$ be the inclusion of the complement. Then there is an adjunction
\begin{equation}
    h_! :\operatorname{QCoh}(U) \leftrightarrows \operatorname{QCoh}(X): h^*,
\end{equation}
so that $h^! \simeq h^*$ in this case. 
\end{prop}
\begin{proof}
By Proposition \ref{prop:hlowershriek1}(ii) we obtain a counit transformation $\varepsilon : h_! h^* \to \operatorname{id}$ given as the composite
\begin{equation}
   \varepsilon : h_! h^* M \simeq \operatorname{Fib}(M \to k_*k^*M) \to M,
\end{equation}
by base-change, there is also a natural equivalence $\eta: \operatorname{id} \simeq h^*h_!$. We need to verify the zig-zag identities. For $M \in \operatorname{QCoh}(X)$ we first need to show that the second morphism in the composite
\begin{equation}
    h^* M \xrightarrow[]{\sim \eta} h^*h_!h^*M \xrightarrow[]{h^* \circ \varepsilon}h^*M
\end{equation}
is an equivalence. But this follows from $h^*\operatorname{Fib}(M \to k_*k^*M) \xrightarrow[]{\sim} h^*M$, where we used that $h^*k_* \simeq 0$ by base change. For the second zig-zag identity we need to verify that for $N \in \operatorname{QCoh}(U)$, the second morphism in the composite
\begin{equation}
    h_!N \xrightarrow{\sim h_! \circ \eta}h_!h^*h_!N \xrightarrow[]{\varepsilon} h_!N
\end{equation}
is an equivalence. By the base-change of Proposition \ref{prop:hlowershriek1} one has $k^*h_! \simeq 0$. Hence the, second morphism is given as
\begin{equation}
   h_!h^*h_!N \simeq \operatorname{Fib}(h_!N \to k_*k^*h_!N) \simeq h_!N.
\end{equation}
\end{proof}
\subsection{Composition of the functors $f_!$.}\label{subsubsec:compositionf_!}
\begin{defn}\cite[Definition 2.5.6]{vanderPutetale}
Let $X$ be a dagger analytic variety. 
\begin{enumerate}[(i)]
    \item We say that an admissible affinoid covering $\{U_i\}_{i \in \mathcal{I}}$ of $X$ is of \emph{finite type} if, for any $i \in \mathcal{I}$, there are only finitely many $i^\prime \in \mathcal{I}$ with $U_i \cap U_{i^\prime} \neq \emptyset$. 
    \item We say that $X$ is \emph{quasi-paracompact} if $X$ admits an admissible affinoid covering of finite type.
\end{enumerate}
\end{defn}
\begin{lem}\label{lem:paracompact}
Let $X$ be a separated dagger analytic variety. The following are equivalent:
\begin{enumerate}[(i)]
    \item $X$ is quasi-paracompact,
    \item Any admissible affinoid covering of $X$ admits a refinement to a finite type covering.
\end{enumerate}
\end{lem}
\begin{proof}
Let $\{V_j\}_{j \in \mathcal{J}}$ be an arbitrary admissible affinoid covering of $X$ and let $\{U_i\}_{i \in \mathcal{I}}$ be a covering with the property (i). For each $i \in \mathcal{I}$ choose a finite subset $\mathcal{J}_0(i) \subset \mathcal{J}$ such that $\{U_i \cap V_j \}_{j \in \mathcal{J}_0(i)}$ is an admissible cover of $U_i$. We claim that the cover $\{U_i \cap V_j : i \in \mathcal{I}, j \in \mathcal{J}_0(i)\}$ has the required property. Indeed, fixing $(i,j)$, the set $S_{ij} := \{ (i^\prime, j^\prime): U_i \cap U_{i^\prime} \neq \emptyset, j^\prime \in \mathcal{J}_0(i^\prime)\}$ is finite, and one has $(U_i \cap V_j) \cap (U_{i^\prime} \cap V_{j^\prime}) \neq \emptyset$ in this covering only if $(i^\prime, j^\prime) \in S_{ij}$. 
\end{proof}
\begin{defn}\label{defn:partially proper}\cite[Definition 0.4.2]{HuberEtale}
A morphism $f: X \to Y$ of dagger analytic varieties is called \emph{partially proper} if it is separated and there exists an admissible affinoid covering $\{Y_i\}_{i \in \mathcal{I}}$ of $Y$ such that for all $i$ there exists two admissible affinoid coverings $\{X_{ij}\}_{j \in \mathcal{J}(i)}, \{X_{ij}^\prime\}_{j \in \mathcal{J}(i)}$, of $f^{-1}(Y_i)$ such that $X_{ij} \Subset_{Y_i} X_{ij}^\prime$ for all $j$. 
\end{defn}
\begin{rmk}\cite[Definition 0.4.2]{HuberEtale}
A morphism is proper if and only if it is partially proper and quasi-compact. Any relative Stein space \cite[\S2.1]{vanderPutSerre} is partially proper. 
\end{rmk}
\begin{lem}\label{lem:partiallyproperparacompact}
Suppose that $f:X \to Y$ is a partially proper morphism of separated dagger analytic varieties and $Y$ is quasi-paracompact. Then the covering $\{Y_i\}_{i \in \mathcal{I}}$ of $Y$ appearing in Definition \ref{defn:partially proper} may be taken to be of finite type. 
\end{lem}
\begin{proof}
According to the definition of partially proper, there exists an admissible affinoid covering $\{U_k\}_{k \in \mathcal{K}}$ of $Y$ such that for all $k$ there exists two admissible affinoid coverings $\{V_{kl}\}_{l \in \mathcal{L}(k)}, \{V_{kl}^\prime\}_{l \in \mathcal{L}(k)}$, of $f^{-1}(U_k)$ such that $V_{kl} \Subset_{U_k} V_{kl}^\prime$ for all $l$. By Lemma \ref{lem:paracompact} we may choose a finite type refinement $\{Y_i\}_{ \in \mathcal{I}}$ of $\{U_k\}_{k \in \mathcal{K}}$. Fix $i$ and choose $k = k(i)$ such that $Y_i \subset U_k$. Define $X_{il}:= V_{kl} \times_{U_k} Y_i$ and $X_{il}^\prime := V_{kl}^\prime \times_{U_k} Y_i$. By \cite[Lemma 1, \S 9.6.2]{BGR}, one has $X_{il} \Subset_{Y_i} X_{il}^\prime$, and $\{X_{il}\}_{l \in \mathcal{L}(k)}$, $\{X_{il}^\prime\}_{l \in \mathcal{L}(k)}$ are admissible coverings of $f^{-1}(Y_i)$, since coverings in sites are stable under base change.  
\end{proof}
\begin{lem}\label{lem:partiallyproperlongboring}
Suppose $f:X \to Y$ is a partially proper morphism of separated dagger analytic varieties and $Y$ is quasi-paracompact. Then every admissible open subset $U \subset X$ which is quasi-compact over $Y$, is contained in some $Z \in c(X/Y)$. 
\end{lem}
\begin{proof}
According to Lemma \ref{lem:partiallyproperparacompact}, there exists a finite type affinoid covering $\{Y_i\}_{i \in \mathcal{I}}$ such that for each $i$, there exists two coverings $\{X_{ij}\}_{j \in \mathcal{J}(i)}$, $\{X_{ij}\}_{j \in \mathcal{J}(i)}$ such that $X_{ij} \Subset_{Y_i} X^\prime_{ij}$ for all $j \in \mathcal{J}(i)$. Since $U$ is quasi-compact over $Y$ then $f^{-1}(Y_i) \cap U$ is quasi compact for all $i$. In particular there exists a finite subset $\mathcal{J}_0(i) \subset \mathcal{J}(i)$ such that $\{X_{ij}\}_{j \in \mathcal{J}_0(i)}$, $\{X_{ij}^\prime\}_{j \in \mathcal{J}_0(i)}$ cover $f^{-1}(Y_i) \cap U$. Since $X_{ij} \Subset_{Y_i} X_{ij}^\prime$ there exists $\varrho_{ij} \in \sqrt{|K^\times|}$, $\varrho_{ij} < 1$, and a finite affinoid generating system $\{t_k\}_k$ for $\mathcal{O}(X_{ij}^\prime)$ over $\mathcal{O}(Y_{i})$ such that 
\begin{equation}
    X_{ij} \subset Z_{ij} := \{ x \in X_{ij}^\prime : |t_k(x)| < \varrho_{ij} \text{ for all }k\}.
\end{equation}
Note that
\begin{equation}\label{eq:setminusZijdefn}
 X_{ij}^\prime \setminus Z_{ij} = \{ x \in X_{ij}^\prime : |t_k(x)| \ge \varrho_{ij} \text{ for some }k\},
\end{equation}
is a quasi-compact admissible open subset of $X_{ij}$. We also define the affinoid subdomain
\begin{equation}
    \widetilde{Z}_{ij} := \{ x \in X^\prime_{ij} : |t_k(x) | \le \varrho^{1/2}_{ij} \text{ for all } k\}
\end{equation}
note that $X_{ij} \subset \widetilde{Z}_{ij} \Subset_{Y_i} X^\prime_{ij}$. Define also
\begin{equation}
\begin{aligned}
  Z_i := \bigcup_{j \in \mathcal{J}_0(i)} Z_{ij}, &&Z^\prime_i := \bigcup_{j \in \mathcal{J}_0(i)} X^\prime_{ij},
\end{aligned}
\end{equation}
then $f^{-1}(Y_i) \cap U \subset Z_i$. We claim that $f^{-1}(Y_i) \setminus Z_{i^\prime}$ is an admissible open subset and $\{f^{-1}(Y_i) \setminus Z_{i^\prime}, Z_{i^\prime}^\prime \cap f^{-1}(Y_i)\}$ is an admissible covering of $f^{-1}(Y_i)$, for any $i^\prime \in \mathcal{I}$. It suffices to show that $(f^{-1}(Y_i) \setminus Z_{i^\prime}) \cap X_{ij} = \bigcap_{j^\prime \in \mathcal{J}_0(i)}X_{ij} \setminus Z_{i^\prime j^\prime}$ is an admissible open subset and that 
\begin{equation}
\left\{\bigcap_{j^\prime \in \mathcal{J}_0(i)}X_{ij} \setminus Z_{i^\prime j^\prime}, \bigcup_{j^\prime \in \mathcal{J}_0(i)} X_{ij} \cap X_{i^\prime j^\prime}^\prime\right\} 
\end{equation}
is an admissible covering, for each $j \in \mathcal{J}(i)$. Suppose that we have shown that $X_{ij}\setminus Z_{i^\prime j^\prime} \subset X_{ij}$ is admissible and $\{X_{ij}\setminus Z_{i^\prime j^\prime}, X_{ij} \cap X_{i^\prime j^\prime}^\prime\}$ is an admissible covering, for each $j^\prime \in \mathcal{J}_0(i)$. Then $\{ X_{ij} \setminus Z_{i^\prime j^\prime}, X_{ij} \cap Z_{i^\prime}^\prime\}$ would be an admissible covering, $X_{ij} \setminus Z_{i^\prime} = \bigcap_{j^\prime \in \mathcal{J}_0(i)} X_{ij} \setminus Z_{i^\prime j^\prime}$ would be an admissible open subset, and $\{X_{ij} \setminus Z_{i^\prime}, Z^\prime_{i^\prime}\}$ would be an admissible covering since it is refined by $\bigcap_{j^\prime \in \mathcal{J}_0(i)}\{X_{ij} \setminus Z_{i^\prime j^\prime}, Z^\prime_{i^\prime}\}$.

Hence to prove the claim, we have reduced to showing that $X_{ij}\setminus Z_{i^\prime j^\prime} \subset X_{ij}$ is admissible and $\{X_{ij}\setminus Z_{i^\prime j^\prime}, X_{ij} \cap X_{i^\prime j^\prime}^\prime\}$ is an admissible covering. We first note that
\begin{equation}
\begin{aligned}
    X_{ij} \cap X_{i^\prime j^\prime} &\Subset_{X_{ij} \times Y_{i^\prime}} X_{ij} \cap X_{i^\prime j^\prime}^\prime, \\
    X_{ij} \cap \widetilde{Z}_{i^\prime j^\prime} &\Subset_{X_{ij} \times Y_{i^\prime}} X_{ij} \cap X_{i^\prime j^\prime}^\prime,
\end{aligned}
\end{equation}
where the map $X_{ij} \cap X_{i^\prime j^\prime}^\prime \to X_{ij} \times Y$ is given by the composite 
\begin{equation}
X_{ij} \cap X_{i^\prime j^\prime} \subset X_{ij} \times X_{i^\prime j^\prime}^\prime \xrightarrow[]{(\operatorname{id},f)} X_{ij} \times Y.
\end{equation}
Hence by \cite[Lemma 1]{vanderPutSerre}, $\{f^{-1}(Y_i)\setminus X_{i^\prime j^\prime}, X_{i^\prime j^\prime}^\prime\}$ and $\{f^{-1}(Y_i)\setminus \widetilde{Z}_{i^\prime j^\prime}, X_{i^\prime j^\prime}^\prime\}$ are admissible coverings of $f^{-1}(Y_i)$, so $X_{ij} = (X_{ij}\setminus X_{i^\prime j^\prime}) \cup (X_{ij} \cap X_{i^\prime j^\prime}^\prime)$ is an admissible cover of $X_{ij}$. But since $Z_{i^\prime j^\prime} \subset X_{i^\prime j^\prime}$ one has 
\begin{equation}
    X_{ij}\setminus Z_{i^\prime j^\prime} = (X_{ij}\setminus X_{i^\prime j^\prime}) \cup (X_{ij} \cap (X_{i^\prime j^\prime}^\prime \setminus Z_{i^\prime j^\prime}))
\end{equation}
is an admissible covering by admissible open subsets, c.f. \eqref{eq:setminusZijdefn}. Hence $X_{ij}\setminus Z_{i^\prime j^\prime}$ is admissible. Furthermore $\{X_{ij} \setminus Z_{i^\prime j^\prime}, X_{ij} \cap X^\prime_{i^\prime j^\prime}\}$ is an admissible covering since it is refined by $\{X_{ij}\setminus \widetilde{Z}_{i^\prime j^\prime}, X_{ij} \cap X^\prime_{i^\prime j^\prime}\}$. This proves the claim. 

Moreover, considering the intersection of $X_{ij} \setminus Z_{i^\prime j^\prime}$ with the admissible covering $\{X_{ij}\setminus X_{i^\prime j^\prime}, X_{ij} \cap X_{i^\prime j^\prime}^\prime\}$ of $X_{ij}$ shows that the morphism $X_{ij} \setminus Z_{i^\prime j^\prime} \to X_{ij}$ is quasi-compact. This shows that $f^{-1}(Y_i)\setminus Z_{i^\prime j^\prime} \to f^{-1}(Y_i)$ is quasi-compact and hence $f^{-1}(Y_i) \setminus Z_{i^\prime} \to f^{-1}(Y_i)$ is a quasi-compact morphism. We now set
\begin{equation}
    \begin{aligned}
        Z:= \bigcup_{i \in \mathcal{I}}Z_i, && Z^\prime := \bigcup_{i \in \mathcal{I}}Z^\prime_i,
    \end{aligned}
\end{equation}
to finish the proof of the Lemma we need to show that:
\begin{enumerate}[(i)]
\label{item:adm1} \item $\{X \setminus Z, Z^\prime\}$ is an admissible covering of $X$;
    \item The composite $Z^\prime \xrightarrow[]{j} X \xrightarrow[]{f} Y$ is quasi-compact;
    \item $X \setminus Z \to X$ is a quasi-compact morphism.
\end{enumerate}
Recall that we chose the covering $\{Y_i\}_{i \in \mathcal{I}}$ to be of finite type and hence for each $i \in \mathcal{I}$ there is a finite subset $S_i \subset \mathcal{I}$ such that $U_i \cap U_{i^\prime} \neq \emptyset$ if and only if $i^\prime \in S_i$. Hence since $Z_i, Z_i^\prime \subset f^{-1}(Y_i)$, one has for each $i$ that 
\begin{equation}\label{eq: intersectfminus1}
    \begin{aligned}
        f^{-1}(Y_i)\setminus Z = \bigcap_{i^\prime \in S_i} f^{-1}(Y_i) \setminus Z_{i^\prime}, && Z^\prime \cap f^{-1}(Y_i) = \bigcup_{i^\prime \in S_i}Z^\prime_{i^\prime} \cap f^{-1}(Y_i), 
    \end{aligned}
\end{equation}
this shows that $X \setminus Z$ and $Z^\prime$ are admissible. Moreover we showed above that $f^{-1}(Y_i) \setminus Z_{i^\prime} \to f^{-1}(Y_i)$ is a quasi-compact morphism for each $i^\prime$, so by \eqref{eq: intersectfminus1} we see that $f^{-1}(Y_i) \setminus Z \to f^{-1}(Y_i)$ is a quasi-compact morphism and hence $X\setminus Z \to X$ is a quasi-compact morphism. 

We have shown that for each $i^\prime$, $\{f^{-1}(Y_i) \setminus Z_{i^\prime}, Z^\prime_{i^\prime}\}$ is an admissible cover. Then $\{f^{-1}(Y_i)\setminus Z_{i^\prime},Z^\prime \cap f^{-1}(Y_i)\}$ is an admissible cover and hence $\{f^{-1}(Y_i)\setminus Z, Z^\prime \cap f^{-1}(Y_i)\}$ is an admissible covering since it is refined by $\bigcap_{i^\prime \in S_i} \{ f^{-1}(Y_i) \setminus Z_{i^\prime}, Z^\prime \cap f^{-1}(Y_i)\}$. This shows that $\{X \setminus Z, Z^\prime\}$ is an admissible covering. 

By separatedness of $Y$, plus the fact that quasi-compact morphisms are stable under base change, $f^{-1}(Y_i) \cap f^{-1}(Y_{i^\prime}) \to f^{-1}(Y_{i^\prime})$ is a quasi-compact morphism. Taking the preimage of $Z_{i^\prime}^\prime$ we see that $Z_{i^\prime}^\prime \cap f^{-1}(Y_i)$ is quasi-compact. By \eqref{eq: intersectfminus1} this shows that $Z^\prime \cap f^{-1}(Y_i)$ is quasi-compact and hence that the composite $Z^\prime \to X \to Y$ is quasi-compact. 
\end{proof}
\begin{cor}
Let $f:X \to Y$ be a partially proper morphism of dagger analytic varieties, with $X$ and $Y$ separated and $Y$ quasi-paracompact. Then $M \in \operatorname{QCoh}(X)$ is compactly supported (relative to $Y$), if and only if there is an admissible open subset $U \subset X$, which is quasi-compact over $Y$, such that $i^*M \simeq 0$, where $i: X \setminus U \to X$ is the inclusion\footnote{By separatedness, $U \to X$ is a quasi-compact morphism, and hence since $X$ is separated then $X \setminus U \subset X$ is an admissible open subset.}. 
\end{cor}
\begin{cor}\label{cor:exhaustiblecriterion}
Let $f: X \to Y$ be a partially proper morphism of dagger analytic varieties with $X$ and $Y$ separated and $Y$ quasi-paracompact. Then $f$ is exhaustible.
\end{cor}
\begin{proof}
We need to show that $\{ Z : Z \in c(X/Y)\}$ is an admissible covering of $X$. Take an admissible covering $\{U_i\}_{i \in \mathcal{I}}$ of $X$. By separatedness, each $U_i \to X$ is a quasi-compact morphism. Hence by Lemma \ref{lem:partiallyproperlongboring}, there exists $Z_i \in c(X/Y)$ with $U_i \subseteq Z_i$. Hence $\{ Z : Z \in c(X/Y)\}$ is refined by an admissible affinoid covering, and so is admissible. 
\end{proof}
\begin{cor}\label{cor:lowershriekcompose}
    Let $g:X_1 \to X_2$ and $f: X_2 \to X_3$ be morphisms of dagger analytic varieties. Assume that $fg$ is exhaustible. Then there is a natural equivalence $(fg)_! \simeq f_! g_!$ in either of the following situations:
    \begin{enumerate}[(i)]
        \item $f$ is quasi-compact;
        \item $f$ is partially proper, $X_2$ and $X_3$ are separated and $X_3$ is quasi-paracompact. 
    \end{enumerate}
\end{cor}
\begin{proof}

    We first note that if $M \in \operatorname{QCoh}(X_1)$ is compactly supported (relative to $X_3$) then it is compactly supported relative to $X_2$. Indeed, suppose we are given $Z_1 \in c(X_1/X_3)$; it suffices to show that $Z_1 \to X_1 \to X_2$ is quasi-compact. However this follows from the fact that $f: X_2 \to X_3$ is separated and the composite $Z_1 \to X_1 \to X_2 \to X_3$ is quasi-compact. 

    We claim that if $M \in \operatorname{QCoh}(X_1)$ is compactly supported (relative to $X_3$) then $g_*M \in \operatorname{QCoh}(X_2)$ is compactly supported (relative to $X_3$). If $f$ is a quasi-compact morphism this is automatic. 
    
    Let $Z_1 \in c(X_1/X_3)$ be such that $i^*M \simeq 0$, where $i: X_1 \setminus Z_1 \to X_1$ is the inclusion. To prove the claim, by Lemma \ref{lem:partiallyproperlongboring} it suffices to show that there is an admissible open subset $U_2 \subset X_2$ with $g(Z_1) \subset U_2$, such that the composite $U_2 \to X_2 \to X_1$ is quasi-compact; indeed then one has 
    \begin{equation}
\begin{tikzcd}
	{X_1 \setminus Z_1 } & {X_1\setminus g^{-1}(U_2)} & {X_2 \setminus U_2} \\
	& {X_1} & {X_2}
	\arrow[hook, from=1-2, to=1-1]
	\arrow["{t^\prime}"', hook, from=1-2, to=2-2]
	\arrow["g", from=2-2, to=2-3]
	\arrow["{g^\prime}", from=1-2, to=1-3]
	\arrow["t"', hook, from=1-3, to=2-3]
	\arrow[from=1-1, to=2-2]
	\arrow["\lrcorner"{anchor=center, pos=0.125}, draw=none, from=1-2, to=2-3]
\end{tikzcd}
    \end{equation}
and hence by base change, one has $t^*g_*M \simeq 0$. 

So let us construct the required admissible open subset $U_2$. Let $\{X_{3,i}\}_{i \in \mathcal{I}}$ be a finite type admissible covering of $X_3$ by affinoid subdomains, and for each $i \in \mathcal{I}$ let $\{X_{2,ij}\}_{j \in \mathcal{J}(i)}$ be an admissible affinoid covering of $f^{-1}(X_{3,i})$. Since $Z_1 \in c(X_3/X_1)$ we can find an admissible open subset $Z_1^\prime \supset Z_1$ such that $Z_1^\prime$ is quasi-compact over $X_3$. Then there is a finite subset $\mathcal{J}_0(i) \subset \mathcal{J}(i)$ such that $Z_1 \cap (fg)^{-1}(X_{3,i}) \subset \bigcup_{j\in \mathcal{J}_0(i)} g^{-1}(X_{2,ij})$, or equivalently 
\begin{equation}
    g(Z_1) \cap f^{-1}(X_{3,i}) \subset \bigcup_{j \in \mathcal{J}_0(i)} X_{2,ij},
\end{equation}
now set $U_{2,i} := \bigcup_{j \in \mathcal{J}_0(i)} X_{2,ij}$ and $U_2 := \bigcup_{i \in \mathcal{I}}U_{2,i}$. Since $X_3$ is separated, for any $i,i^\prime \in \mathcal{I}$ the morphism $f^{-1}(X_{3,i}) \cap f^{-1}(X_{3,i^\prime}) \to f^{-1}(X_{3,i})$ is quasi-compact and hence $U_{2,i^\prime} \cap f^{-1}(X_{3,i})$ is quasi-compact for any $i,i^\prime$. Since the covering $\{X_{3,i}\}_{i \in \mathcal{I}}$ is of finite type, for each $i$ the set $S_i := \{i^\prime \in \mathcal{I} : U_i \cap U_{i^\prime} \neq \emptyset \}$ is finite. Therefore one has 
\begin{equation}
    U_2 \cap f^{-1}(X_{3,i}) = \bigcup_{i^\prime \in S_i} U_{2,i^\prime} \cap f^{-1}(X_{3,i}),
\end{equation}
which is quasi-compact, therefore $U_2$ is quasi-compact over $X_3$ as required. 

Now I will explain how these statements lead to the proof of the Corollary. The functor $f_!g_!$ is colimit-preserving. Since $fg$ is exhaustible, every $M \in \operatorname{QCoh}(X_1)$ may be written as a colimit $M \simeq \varinjlim N$ of compactly supported (relative to $X_3$) objects $N$. Hence, $f_! g_! M \simeq \varinjlim f_! g_!N$. Now $g_!N = g_*N$, because $N$ is compactly supported (relative to $X_2$), by the first observation. And $f_!g_*N = f_* g_*N$, because $g_*N$ is compactly supported (relative to $X_3$), by the claim. Hence,
\begin{equation}
f_!g_!M \simeq \varinjlim_{N \to M} f_*g_*N, 
\end{equation}
so that $f_! g_!$ identifies with the left Kan extension of $f_*g_* \simeq (fg)_*$ from the category of compactly supported (relative to $X_3$) objects. Hence, the desired equivalence $f_!g_! \simeq (fg)_!$ follows from Kan extension. 
\end{proof}
\subsection{Grothendieck duality}\label{sec:grothdual}
Let $f:X \to Y$ be any morphism of dagger analytic varieties. The functor $f^!$ was defined via the adjoint functor theorem for presentable $\infty$-categories, and thus far the only situation in which we have a good understanding of it is that of Proposition \ref{prop:hlowershriekadjunction}. In this section we try to understand the functor $f^!$ for more general morphisms. Let $1_{Y} \in \operatorname{QCoh}(Y)$ be the tensor unit. By the projection formula, and the counit $f_!f^! \to \operatorname{id}$, there is a morphism
\begin{equation}
    f_! (f^!1_Y\widehat{\otimes}_X f^*N) \cong f_!f^!1_Y\widehat{\otimes}_YN  \to N
\end{equation}
for any $N \in \operatorname{QCoh}(Y)$, which induces by adjunction a morphism
\begin{equation}\label{eq:grothdualmorph}
f^!1_Y\widehat{\otimes}_X f^*  \to f^!.
\end{equation}
The investigation of the conditions on $f$ which lead to \eqref{eq:grothdualmorph} being an isomorphism, is usually known as Grothendieck duality theory. We first prove a Grothendieck duality theorem for the projection from the open polydisk. The arguments of this section are very strongly inspired by those of \cite[Lecture VII, Lecture VIII]{CondensedComplexGeometry}. Key to this section is the Ind-Banach algebra
\begin{equation}
    K\langle T^{-1}, T]^\dagger := \left\{\sum_{n= - \infty}^\infty a_n T^n: \begin{array}{c} a_n \in K,  a_n = 0 \text{ for }n \gg 0, \\ a_n r^{n} \to 0 \text { as } n \to -\infty \text{ for some } r < 1. \end{array}\right\}.
\end{equation}
Let $B$ be a a dagger affinoid algebra. Let\footnote{See \cite[\S 4]{BenBassatAnalytification} for the definitions of the natural bornologies on $K[T] , K[\![T]\!], K(\!(T)\!)$, etc.} 
\begin{equation}
\begin{aligned}
    B[T] := B \widehat{\otimes}_K K[T],  &&  B \langle T^{-1}, T]^\dagger := B \widehat{\otimes}_K K \langle T^{-1}, T]^\dagger, && \text{etc.}
\end{aligned}
\end{equation} 
There is a forgetful functor 
\begin{equation}
    \operatorname{QCoh}(\operatorname{Sp}(B\langle T \rangle^\dagger)) = \operatorname{Mod}_{B\langle T \rangle^\dagger}\operatorname{QCoh}(\operatorname{Sp}(K)) \to \operatorname{Mod}_{B[T]}\operatorname{QCoh}(\operatorname{Sp}(K))
\end{equation}
which is fully-faithful, because $B[T] \to B\langle T \rangle^\dagger$ is a homotopy epimorphism \cite[Theorem 5.8]{BenBassatAnalytification}.
\begin{lem}\label{lem:formalLaurent}
With notations as above. If $M \in  \operatorname{QCoh}(\operatorname{Sp}(B\langle T \rangle^\dagger))$ dies on base-change to $B\langle T^{-1}, T\rangle$ then, when viewed as a $B[T]$-module, $M$ dies on base-change to $B \langle T^{-1}, T]^\dagger$. 
\end{lem}
\begin{proof}
By the associativity properties of $\widehat{\otimes}^\mathbf{L}$, the natural morphism
\begin{equation}
  B\langle T^{-1}, T]^\dagger \widehat{\otimes}_{B[T]}^\mathbf{L}M \xrightarrow[]{\sim} B\langle T^{-1}, T]^\dagger \widehat{\otimes}^\mathbf{L}_{B[T]} B \langle T \rangle^\dagger\widehat{\otimes}_{B\langle T\rangle^\dagger}^\mathbf{L} M  
\end{equation}
is an equivalence. In order to conclude it suffices to show that the canonical morphism 
\begin{equation}
    B\langle T^{-1}, T]^\dagger \widehat{\otimes}^\mathbf{L}_{B[T]} B \langle T \rangle^\dagger \to B\langle T^{-1}, T \rangle^\dagger
\end{equation}
is an equivalence. The object  $B\langle T^{-1}, T]^\dagger$, viewed as a $B[T]$-module, admits a flat resolution by the two-term complex 
\begin{equation}
    \Big[B\langle S \rangle^\dagger \widehat{\otimes}_K K[T] \xrightarrow[]{\times (1-ST)}  B\langle S \rangle^\dagger \widehat{\otimes}_K K[T] \Big],
\end{equation}
so that $ B\langle T^{-1}, T]^\dagger \widehat{\otimes}^\mathbf{L}_{B[T]} B \langle T \rangle^\dagger$ is represented by the two-term complex
\begin{equation}
    \Big[B\langle S, T \rangle^\dagger \xrightarrow[]{\times (1-ST)}  B\langle S, T \rangle^\dagger \Big],
\end{equation}
which is a resolution of $ B\langle T^{-1}, T \rangle^\dagger$. 
\end{proof}
\begin{lem}\label{lem:uppershriektransformation}
Let $f: X \to Y$ be an exhaustible morphism of dagger-analytic varieties and let $t: Y^\prime \to Y$ be a transversal morphism. Let $t^\prime, f^\prime$ be determined by the Cartesian square 
\begin{equation}
\begin{tikzcd}[cramped]
	{X^\prime} & {Y^\prime} \\
	X & Y
	\arrow["{f^\prime}", from=1-1, to=1-2]
	\arrow["{t^\prime}"', from=1-1, to=2-1]
	\arrow["\ulcorner"{anchor=center, pos=0.125}, draw=none, from=1-1, to=2-2]
	\arrow["t", from=1-2, to=2-2]
	\arrow["f"', from=2-1, to=2-2]
\end{tikzcd}
\end{equation}
Then there is a natural morphism 
\begin{equation}
    t^{\prime,*} f^! \to f^{\prime,!} t^*.
\end{equation}
\end{lem}
\begin{proof}
The desired morphism is adjoint to $f^\prime_! t^{\prime,*} f^! \simeq t^*f_! f^! \to t^*$, where the first morphism is base-change (Proposition \ref{prop:lowershriekbasechange}), and the second is induced by the counit of $f^! \dashv f^!$.
\end{proof}
\begin{prop}\label{prop:Grothdual1dim}
Let $Y$ be a dagger-analytic variety and let $f: X = \mathring{\mathbb{D}}^1_K \times Y \to Y$ be the projection off the dagger unit disk. Then:
\begin{enumerate}[(i)]
    \item $f^!$ commutes with restrictions. That is, if $t: U \to Y $ is the inclusion of an admissible open subset then the natural morphism $t^{\prime,*} f^! \to f^{\prime,!} t^*$ is an equivalence. 
    \item The natural morphism $f^! 1_Y \widehat{\otimes}_X f^* \to f^!$ is an equivalence;
    \item There is an equivalence $f^!1_Y \simeq 1_X[1]$. 
\end{enumerate}
\end{prop}
\begin{proof}
As in the proof of \cite[Theorem 12.17]{CondensedComplexGeometry}, in order to prove the claims (i) and (ii) one may use descent to reduce to the case when $Y = \operatorname{Sp}(B)$ is a dagger affinoid. Then one factors $f$ as $f = gj$ where $g: \mathbb{D}^1_K \times \operatorname{Sp}(B) \to \operatorname{Sp}(B)$ is the projection. By Proposition \ref{prop:hlowershriekadjunction}, the functor $j_!$ is fully-faithful. By the previous discussion the forgetful functor from $\operatorname{QCoh}(\mathbb{D}^1_K \times \operatorname{Sp}(B)) = \operatorname{Mod}_{B\langle T\rangle^\dagger}(\operatorname{QCoh}(\operatorname{Sp}(K)))$ to $ \operatorname{Mod}_{B [ T ]}(\operatorname{QCoh}(\operatorname{Sp}(K)))$ is fully-faithful.

Let $M \in \operatorname{QCoh}(X)$. By base-change, c.f. Proposition \ref{prop:lowershriekbasechange}, $j_!M$ dies on base-change to $B\langle T^{-1}, T \rangle^\dagger$. Hence by Lemma \ref{lem:formalLaurent} above, $j_!M$ (viewed as a $B[T]$-module) dies on base-change to $B\langle T^{-1}, T ]^\dagger$. In particular $j_!M$ (viewed as a $B[T]$-module) dies on base-change to $B(\!(T^{-1})\!)$.

Now let $N \in \operatorname{QCoh}(\operatorname{Sp}(B))$. Using only that $M$ dies on base-change to $B(\!(T^{-1})\!)$, so that any morphism from $M$ to a $B(\!(T^{-1})\!)$-module vanishes, we now calculate as in the proof of \cite[Theorem 12.17]{CondensedComplexGeometry}:
\begin{equation}
\begin{aligned}
R\underline{\operatorname{Hom}}_{\mathring{\mathbb{D}}^1_K \times Y}(M, f^!N) &\simeq R\underline{\operatorname{Hom}}_{B}(j_! M, N)  \\
&\simeq R\underline{\operatorname{Hom}}_{B[T]}(j_!M, R\underline{\operatorname{Hom}}_B(B[T],N)) \\ 
&\simeq R\underline{\operatorname{Hom}}_{B[T]}(j_!M, N(\!(T^{-1})\!)/N[T]) \\
&\simeq R\underline{\operatorname{Hom}}_{B[T]}(j_!M, N[T])[1].
\end{aligned}
\end{equation}
Now using that $j_!M$ (viewed as a $B[T]$)-module) dies on base change to $B\langle T^{-1}, T]^\dagger$, there is an equivalence 
\begin{equation}
   R\underline{\operatorname{Hom}}_{B[T]}(j_!M, N[T]) \simeq  R\underline{\operatorname{Hom}}_{B[T]}(j_!M, \operatorname{Fib}(B[T] \to B\langle T^{-1}, T]^\dagger)\widehat{\otimes}^\mathbf{L}_B N). 
\end{equation}
However, the natural morphism
\begin{equation}
     \operatorname{Fib}(B[T] \to B\langle T^{-1}, T]^\dagger) \to \operatorname{Fib}(B\langle T\rangle^\dagger \to B\langle T^{-1}, T\rangle^\dagger)
\end{equation}
is an equivalence as  $B\langle T^{-1}, T]^\dagger/B[T] \cong B\langle T^{-1}, T\rangle^\dagger/B\langle T\rangle^\dagger$. Hence, we may calculate further: 
\begin{equation}
    \begin{aligned}
    R\underline{\operatorname{Hom}}_{B[T]}(j_!M, N[T]) &\simeq  R\underline{\operatorname{Hom}}_{B[T]}(j_!M, \operatorname{Fib}(B[T] \to B\langle T^{-1}, T]^\dagger)\widehat{\otimes}^\mathbf{L}_B N) \\
    &\simeq R\underline{\operatorname{Hom}}_{B[T]}(j_!M, \operatorname{Fib}(B\langle T\rangle^\dagger \to B\langle T^{-1}, T\rangle^\dagger) \widehat{\otimes}^\mathbf{L}_B N) \\
    &\simeq R\underline{\operatorname{Hom}}_{B\langle T \rangle^\dagger}(j_!M, \operatorname{Fib}(B\langle T\rangle^\dagger \to B\langle T^{-1}, T\rangle^\dagger) \widehat{\otimes}^\mathbf{L}_B N) \\
    &\simeq R\underline{\operatorname{Hom}}_{B\langle T \rangle^\dagger}(j_!M, B\langle T\rangle^\dagger \widehat{\otimes}^\mathbf{L}_B N) \\
    &\simeq R \underline{\operatorname{Hom}}_{\mathring{\mathbb{D}}^1_K \times Y}(M, f^*N). 
    \end{aligned}
\end{equation}
In the third line we used fully-faithfulness and in the fourth line we used that $j_!M$ dies on base-change to $B \langle T^{-1}, T\rangle^\dagger$. This gives $f^! \simeq f^*[1]$ via an equivalence compatible with restrictions. 

For (iii), let $Y$ again be arbitrary. We note that by Lemma \ref{lem:uppershriektransformation} applied to the square 
\begin{equation}
\begin{tikzcd}[cramped]
	{\mathring{\mathbb{D}}^1_K\times Y} & Y \\
	{\mathring{\mathbb{D}}^1_K} & {*}
	\arrow["{f= p^\prime}", from=1-1, to=1-2]
	\arrow["{\pi^\prime}"', from=1-1, to=2-1]
	\arrow["\ulcorner"{anchor=center, pos=0.125}, draw=none, from=1-1, to=2-2]
	\arrow["\pi", from=1-2, to=2-2]
	\arrow["p"', from=2-1, to=2-2]
\end{tikzcd}
\end{equation}
there is a natural transformation $\pi^{\prime,*} p^! \to f^! \pi^*$. Plugging in the unit object we obtain $\pi^{\prime,*}p^!K \to f^!1_Y$. By using the equivalence $p^!K \simeq 1_{\mathring{\mathbb{D}}^1_K}[1]$ coming from (ii), we obtain a morphism $1_X[1] \to f^!1_Y$ which is compatible with restrictions by (i). Hence we may invoke descent, and conclude that this is an equivalence by (ii).  
\end{proof}
\begin{cor}\label{cor:Grothdualndim}
Let $Y$ be a dagger-analytic variety and let $f: X = \mathring{\mathbb{D}}^n_K \times Y \to Y$ be the projection off the $n$-dimensional open unit disk. Then:
\begin{enumerate}[(i)]
    \item $f^!$ commutes with restrictions;
    \item The natural morphism $f^!1_Y \widehat{\otimes}_X f^* \to f^!$ is an equivalence, 
    \item There is an equivalence $f^!1_Y \simeq 1_X[n]$.
\end{enumerate}
\end{cor}
\begin{proof}
As in the proof of Proposition \ref{prop:Grothdual1dim} above, to prove claims (i) and (ii) one may use descent to reduce to the case when $Y = \operatorname{Sp}(A)$ is affinoid, and then both claims follow by factoring $f$ as the composite
\begin{equation}
    \mathring{\mathbb{D}}^n_K \times Y \to \mathring{\mathbb{D}}^{n-1}_K \times Y \to \dots \to \mathring{\mathbb{D}}^1_K \times Y \to Y
\end{equation}
and repeatedly applying Proposition \ref{prop:Grothdual1dim}. In this situation the $!$-functors can be seen to compose by Corollary \ref{cor:lowershriekcompose}.

(iii): By considering the factorization of $p: \mathring{\mathbb{D}}^n_K \to *$ as $\mathring{\mathbb{D}}^n_K \to \mathring{\mathbb{D}}^{n-1}_K \to \dots \to \mathring{\mathbb{D}}^1_K \to *$ and repeatedly applying Proposition \ref{prop:Grothdual1dim} we arrive at the equivalence  $p^!K \simeq 1_{\mathring{\mathbb{D}}^n_K}[n]$; here we again appeal to Corollary \ref{cor:lowershriekcompose} to ensure that the $!$-functors compose. Now we may proceed in an entirely similar way to the proof of Proposition \ref{prop:Grothdual1dim}(iii) to deduce the equivalence $f^!1_Y \simeq 1_X[n]$ for arbitrary dagger-analytic varieties $Y$.
\end{proof}
\begin{lem}
Let $j: X \to Y$ be an open immersion of dagger analytic varieties. Assume that $j$ is exhaustible. There is a canonical morphism $j^! \to j^*$.
\end{lem}
\begin{proof}
The desired morphism is the composite $j^! \to j^!j_*j^* \simeq j^*$, where the first morphism is obtained from the unit of $j^* \dashv j_*$ and the second is base-change. 
\end{proof}
\begin{lem}\label{lem:partiallyproperopen} 
Let $j: X \to Y$ be an open immersion of dagger analytic varieties which is exhaustible and partially-proper. Then, the canonical morphism $j^! \to j^*$ is an equivalence. 
\end{lem}
\begin{proof}
By descent, one may reduce to the case when $Y$ is affinoid. The correspondence between open subsets of the Berkovich space and partially-proper open immersions \cite[Lemma 8.1.7]{HuberEtale}, together with \cite[Lemma 8.1.8(iii)]{HuberEtale}, implies that $X = \bigcup_{i \in \mathcal{I}} U_i$ where $j_i: U_i \hookrightarrow Y$ are admissible open subsets of the form $U_i = Y \setminus V_i$ where $V_i \subseteq Y$ are quasi-compact admissible open subsets. For such subsets we know that $j_i^! \xrightarrow{\sim} j_i^*$, by Proposition \ref{prop:hlowershriekadjunction} (see also Remark \ref{rmk:exhaustiblesufficient}). Hence by descent on $X$ (appealing to Corollary \ref{cor:lowershriekcompose} to ensure that $!$-functors compose), we conclude that $j^! \xrightarrow[]{\sim} j^*$. 
\end{proof}
The following Lemma of course follows from the previous, but since the proof is much simpler in this case we decided to include it. 
\begin{lem}\label{lem:openclosedimmersion}
Let $j: X \to Y$ be an immersion of dagger analytic varieties which is both open and closed. Then the canonical morphism $j^! \to j^*$ is an equivalence. 
\end{lem}
\begin{proof}
Because $j$ is both open and closed, we may write $j: X \to X \sqcup Z = Y$ where $Z$ is the (open and closed) complement of $X$ in $Y$. Let $k: Z \to Y$ be the inclusion. By descent, there is an equivalence \begin{equation}\label{eq:disjointj^!}
    \operatorname{id} \simeq j_*j^* \oplus k_*k^*.
\end{equation} 
By base-change, $j^!j_* \simeq \operatorname{id}$ and $j^!k_* \simeq 0$. Hence, applying $j^!$ to both sides of \eqref{eq:disjointj^!} gives $j^! \xrightarrow{\sim} j^*$.
\end{proof}
\begin{lem} \label{lem:separatedetale1}
Let $f: X \to Y$ be a separated \'etale morphism of dagger analytic varieties. Then:
\begin{enumerate}[(i)]
    \item Assume that $f$ is exhaustible. There is a canonical morphism $f^! \to f^*$. 
    \item If $f$ is finite then the canonical morphism $f^* \to f^!$ is an equivalence compatible with restrictions.
\end{enumerate}
\end{lem}
\begin{proof}
(i): Because $f$ is separated and \'etale, its diagonal $\Delta: X \to X \times_Y X$ is an open and closed immersion. Hence by Lemma \ref{lem:openclosedimmersion}, there is a canonical equivalence $\Delta^! \xrightarrow[]{\sim} \Delta^*$. Let $\pi_1, \pi_2 : X \times_Y X$ be the projections. Because $f$ is exhaustible, and transversal by Proposition \ref{prop:smoothbasechange}, there is a canonical morphism $\pi_2^*f^! \to \pi_1^! f^*$. Hence we obtain a morphism $f^! \to f^*$ via the composite 
\begin{equation}
    f^! \simeq \Delta^* \pi_2^* f^! \simeq \Delta^! \pi_2^*f^! \to \Delta^! \pi_1^! f^* \simeq f^*.  
\end{equation}
(ii): In order to prove the claim we may use descent to reduce to the case when $Y = \operatorname{Sp}(B)$ is affinoid, in which case (since $f$ is finite) then $X = \operatorname{Sp}(A)$ is also affinoid and $f$ is induced by a finite \'etale morphism $f^\# : B \to A$. Let $M \in \operatorname{QCoh}(Y)$. Since $A$ is finite projective as a $B$-module the canonical morphism $R\underline{\operatorname{Hom}}_B(A,B) \widehat{\otimes}_B^\mathbf{L} M \to R\underline{\operatorname{Hom}}_B(A,M)$ is an equivalence and $R\underline{\operatorname{Hom}}_B(A,B) \simeq \underline{\operatorname{Hom}}_B(A,B)$. Hence the canonical morphism $f^!M \to f^*M$ is seen to be equivalent to a morphism
\begin{equation}
\underline{\operatorname{Hom}}_B(A,B)\widehat{\otimes}_B^\mathbf{L} M \to A \widehat{\otimes}_B^\mathbf{L} M.
\end{equation}
The morphism $\underline{\operatorname{Hom}}_B(A,B) \to A$ is seen to be an isomorphism because of the classical theory of the trace for finite \'etale morphisms: it has an inverse sending $1 \in A$ to $\operatorname{Tr}_{f^\#}: A \to B$. This gives an equivalence $f^!M \xrightarrow[]{\sim} f^*M$ compatible with restrictions to affinoid subdomains. 
\end{proof}
\begin{prop}\label{prop:partiallyproperetale}
Let $f: X \to Y$ be an exhaustible and partially-proper \'etale morphism of dagger-analytic varieties. Then the canonical morphism $f^! \to f^*$ is an equivalence.
\end{prop}
\begin{proof}
By descent we may reduce to the case when $Y$ is affinoid. According to \cite[Proposition 8.3.4]{HuberEtale} we may equivalently regard $f$ as an \'etale morphism of strictly Hausdorff $K$-analytic spaces $f_{\mathrm{max}}: X_{\mathrm{max}} \to Y_{\mathrm{max}}$. According to \cite[Definition 4.2.19]{temkin_introduction_2011} every such morphism is locally on the source and target (in the Berkovich topology on $X_{\mathrm{max}}$ and $Y_{\mathrm{max}}$) finite \'etale. Hence using the correspondence between open subsets of the Berkovich space and partially-proper open immersions \cite[Lemma 8.1.7]{HuberEtale}, and descent on $Y$, we reduce to the following situation: there exists a cover $\{X_i\}_{i \in \mathcal{I}}$ of $X$ by partially-proper open immersions $j_i: X_i \to X$ such that each $f_i: X_i \to Y$ is finite \'etale. We note that each $j_i$ is exhaustible by Lemma \ref{lem:exhaustibleproperties}(ii). Hence, by Lemma \ref{lem:separatedetale1} each $f_i$ satisfies $f_i^! \xrightarrow[]{\sim} f_i^*$ and by Lemma \ref{lem:partiallyproperopen} each $j_i$ satisfies $j_i^! \xrightarrow[]{\sim} j_i^*$. Hence by descent on $X$ (appealing to Corollary \ref{cor:lowershriekcompose} to ensure that $!$-functors compose), we conclude that $f^! \to f^*$ is an equivalence.
\end{proof}
\begin{thm}
Let $f: X \to Y$ be a smooth, exhaustible and partially-proper morphism of dagger-analytic varieties. Then:
\begin{enumerate}[(i)]
    \item The functor $f^!$ commutes with restrictions;
    \item The canonical morphism $f^! 1_Y \widehat{\otimes}_X f^* \to f^!$ is an equivalence.
\end{enumerate}
\end{thm}
\begin{proof}
In order to prove both assertions we may use descent to reduce to the case when $Y$ is affinoid. Due to \cite[Proposition 2.5.17]{BerkovichSpectral}, the partially-properness of $f$ is equivalent to $f_{\mathrm{max}} : X_{\mathrm{max}} \to Y_{\mathrm{max}}$ being a \emph{closed} morphism of strictly Hausdorff $K$-analytic spaces. It is also \emph{quasi-smooth} and hence by \cite[Corollary 5.4.8]{DucrosBerkovich}, $f_{\mathrm{max}} : X_{\mathrm{max}} \to Y_{\mathrm{max}}$ is a \emph{smooth} morphism of strictly Hausdorff $K$-analytic spaces. Hence, $f_\mathrm{max}$ factors (locally in the Berkovich topology on $X_{\mathrm{max}}$) as an \'etale morphism followed by the projection off the affine line. By using the correspondence between \'etale morphisms of strictly Hausdorff $K$-analytic spaces and partially-proper \'etale morphisms \cite[Proposition 8.3.4]{HuberEtale} we see that there is a cover $\{X_i\}_{i \in \mathcal{I}}$ of $X$ by partially-proper open immersions $j_i: X_i \to X$ such that each restriction $f_i: X_i \to Y$ factors as $X_i \to \mathbb{A}^{n_i}_Y \to Y$, where $X_i \to \mathbb{A}^{n_i}_Y$ is partially-proper and \'etale, and $\mathbb{A}^{n_i}_Y \to Y$ is the projection. By Corollary \ref{cor:exhaustiblecriterion}, each $X_i \to \mathbb{A}^{n_i}_Y$ is exhaustible, and $\mathbb{A}^{n_i}_Y \to Y$ is exhaustible, hence $f_i$ is exhaustible. Then by Lemma \ref{lem:exhaustibleproperties}, each $j_i$ is also exhaustible. 

Both claims (i) and (ii) hold for the projection $\mathbb{A}^{m}_Y \to Y$. Indeed, one considers the covering of $\mathbb{A}^{m}_Y$ by open disks $k_\varrho : \mathring{\mathbb{D}}^m_Y(\varrho) \to \mathbb{A}^{m}_Y$ of increasing radius $\varrho \in \sqrt{|K|^\times}$. Each $k_\varrho$ satisfies $k_\varrho^! \simeq k_\varrho^*$ by Proposition \ref{prop:hlowershriekadjunction}, so that by Corollary \ref{cor:Grothdualndim} and descent (on $\mathbb{A}^{m}_Y$) both (i) and (ii) are true for this morphism. 

Hence by Proposition \ref{prop:partiallyproperetale} both claims (i) and (ii) hold for the morphisms $f_i: X_i \to Y$ (appealing to Corollary \ref{cor:lowershriekcompose} to ensure that $!$-functors compose). Now by Lemma \ref{lem:partiallyproperopen} above, each $j_i$ satisfies $j_i^! \xrightarrow[]{\sim} j_i^*$ and hence by descent on $X$ (appealing to Corollary \ref{cor:lowershriekcompose} to ensure that $!$-functors compose) we deduce that both claims (i) and (ii) hold for $f$. 
\end{proof}
\begin{rmk}\label{rmk:final}
\begin{enumerate}[(i)]
    \item Every smooth proper morphism $f: X \to Y$ of dagger-analytic varieties, is automatically exhaustible and partially proper. Indeed, a morphism of dagger-analytic varieties is proper if and only if it is partially-proper and quasi-compact, and every qcqs morphism is exhaustible. 
    \item By Corollary \ref{cor:exhaustiblecriterion}, we recall that every smooth and partially proper morphism $f: X \to Y$ of dagger-analytic varieties with $Y$ quasi-paracompact and separated, is exhaustible. In particular when $Y$ is affinoid, this includes the case of relative Stein spaces \cite{vanderPutSerre}. 
\end{enumerate}

\end{rmk}

\Addresses

\end{document}